\documentclass[11pt, reqno]{amsart}

\usepackage{fullpage}
\usepackage[utf8]{inputenc}
\usepackage{amsmath}
\usepackage{amsfonts}
\usepackage{amssymb}
\usepackage{amsthm}
\usepackage{enumerate}
\usepackage{xcolor}
\usepackage{hyperref}
\usepackage{thmtools}
\usepackage[english]{babel}
\usepackage{csquotes}
\usepackage{biblatex}
\usepackage{makecell}
\usepackage{longtable}

\numberwithin{equation}{section}

\declaretheorem[
    name=Theorem,
	numberwithin=section,
	]{thm}

\declaretheorem[
    name=Proposition,
	numberwithin=section,
	]{prop}

\declaretheorem[
    name=Definition,
	style=definition,
	numberwithin=section,
	]{defin}

\declaretheorem[
	name=Lemma,
	numberwithin=section
	]{lem}

\declaretheorem[
	name=Corollary,
	numberwithin=section
	]{cor}

\declaretheorem[
	name=Remark,
	style=remark,
	numberwithin=section
	]{rem}

\newcommand*{\C}{\mathbb{C}}
\newcommand*{\Z}{\mathbb{Z}}
\newcommand*{\R}{\mathbb{R}}
\newcommand*{\B}{\mathbb{B}}
\newcommand*{\Y}{\mathbb{Y}}
\newcommand*{\Half}{\mathbb{H}}
\renewcommand*{\S}{\mathbb{S}}
\newcommand*{\pd}{\partial}
\newcommand*{\eps}{\varepsilon}
\renewcommand*{\phi}{\varphi}
\newcommand*{\mfk}[1]{\mathfrak{#1}}
\renewcommand{\Tilde}{\widetilde}

\DeclareMathOperator{\Real}{Re}
\DeclareMathOperator{\diag}{diag}
\DeclareMathOperator{\id}{id}
\DeclareMathOperator{\tr}{tr}
\newcommand{\Vector}[2]{\begin{pmatrix}#1\\#2\end{pmatrix}}

\title{Mode stability of self-similar wave maps without symmetry in higher dimensions}
\author{Roland Donninger}
\email{roland.donninger@univie.ac.at}
\address{Universit\"at Wien, Fakult\"at f\"ur Mathematik,
  Oskar-Morgenstern-Platz 1, 1090 Vienna, Austria}
\thanks{This research was funded in whole or in part by the
Austrian Science Fund (FWF) 10.55776/P34560,
10.55776/PIN2161424, and 10.55776/PAT9429324. For open access purposes, the authors have
applied a CC BY public copyright license to any author-accepted manuscript version arising from this submission.}

\author{Frederick Moscatelli}
\email{frederick.moscatelli@univie.ac.at}
\address{Universit\"at Wien, Fakult\"at f\"ur Mathematik,
  Oskar-Morgenstern-Platz 1, 1090 Vienna, Austria}

\begin{document}

\begin{abstract}
We consider wave maps from $(1+d)$-dimensional Minkowski space into the $d$-sphere. For every $d \geq 3$, there exists an explicit self-similar solution that exhibits finite time blowup. This solution is corotational and its mode stability in the class of corotational functions is known. Recently, Weissenbacher, Koch, and the first author proved mode stability without symmetry assumptions in $d =3$. In this paper we extend this result to all $d \geq 4$. On a technical level, this is the first successful implementation of the quasi-solution method where two additional parameters are present.

\end{abstract}

\maketitle

\section{Introduction}
\noindent In this paper, we consider wave maps from $(1+d)$-dimensional Minkowski space $\R^{1,d} = \R \times \R^d$ to the $d$-dimensional sphere $\S^d \subseteq \R^{d+1}$. The (extrinsic) wave maps equation for a function $U : \R^{1,d} \to \S^d$ is given by
\begin{align}\label{ex wm eq}
    \pd^\mu \pd_\mu U + \left(\pd^\mu U \cdot \pd_\mu U\right) U =0,
\end{align}
where $\cdot $ denotes the Euclidean inner product and Einstein's summation convention is in effect\footnote{The derivative with respect to the $\mu$-th slot is denoted by $\pd_\mu$ with the conventions $\pd^0 := - \pd_0$ and $\pd^j := \pd_j$ for $j \in \{1,\ldots,d\}$. Greek indices as $\mu$ run from $0$ to $d$ and latin letters as $i,j,k$ run from $1 $ to $d$.}. Eq.~\eqref{ex wm eq} is the simplest instance of a geometric wave equation and has the (formally) conserved energy
\begin{align*}
    E[U](t)= \frac{1}{2} \int_{\R^d} |\pd_0 U(t,x)|^2 + \pd^j U(t,x) \cdot \pd_j U(t,x) dx.
\end{align*}
Furthermore, Eq.~\eqref{ex wm eq} is scaling invariant, i.e., if $U$ solves Eq.~\eqref{ex wm eq}, so does $U_a(t,x) := U\left( \frac{t}{a},\frac{x}{a} \right)$ for $a > 0$. The energy scales as
\begin{align*}
    E[U_a](t) = a^{d-2} E[U]\left( \frac{t}{a} \right),
\end{align*}
which makes Eq.~\eqref{ex wm eq} energy supercritical for $d \geq 3$. In these dimensions we consider an explicit self-similar solution $U_*$ of Eq.~\eqref{ex wm eq},
\begin{align*}
    U_\ast(t,x) = F_\ast\left(  \frac{x}{1 -t} \right),
\end{align*}
where 
\begin{align*}
    F_\ast(\xi) := \frac{1}{d-2 + |\xi|^2} \Vector{2 \sqrt{d-2}\, \xi}{ d- 2 - |\xi|^2}.
\end{align*}
Observe that $U_\ast$ suffers a gradient blowup as $t \nearrow 1$, even though its initial data $(U_\ast(0,\cdot), \pd_0 U_\ast(0,\cdot) )$ are perfectly smooth. The existence of such a solution was first proved in \cite{Sha88} and then found in closed form in \cite{SpeTur90} for $d =3$ and in \cite{BizBie15} for $d \geq 4$.

A natural question is whether this solution is stable under perturbations of the initial data. This was studied extensively under the assumption of corotational symmetry. However, without symmetry assumptions little is known about the stability of $U_\ast$. The first step in a rigorous stability analysis is a proof of \emph{mode stability}. This was achieved recently for $d=3$ in \cite{WeiKocDon25}. In the present paper we establish mode stability of $U_\ast$ for all remaining dimensions $d \geq 4$. Based on this, we prove the full nonlinear asymptotic stability in the companion paper \cite{DonMos26}.

Proving mode stability without symmetry assumptions is challenging because the spectral equation is a system of partial differential equations (PDEs) rather than a single ordinary differential equation (ODE) as is the case in the corotational setting. The strategy of the proof is quickly summarized as follows.
\begin{itemize}
\item Since the blowup solution we perturb around is corotational, it is possible to reduce the spectral PDE to a system of coupled ODEs by expanding in spherical harmonics. This is therefore the first reduction step we perform.
\item Our goal is to apply the quasi-solution method that was key in the proof of mode stability in the corotational case \cite{CosDonGloHua16, CosDonGlo17, Don24}. This method can only be applied to single ODEs and thus, the next and crucial step is to decouple the system of spectral ODEs. For this we employ representation theory of Lie algebras. As was already observed in \cite{WeiKocDon25}, there is a fundamental relation of the coupling operator to a Casimir operator of a certain irreducible representation of the Lie algebra $\mathfrak{so}(d)$. By exploiting this, we achieve the desired decoupling. The fact that we work in general dimensions $d\geq 3$ requires a novel systematic treatment of the decoupling step that goes beyond \cite{WeiKocDon25} where a more straightforward connection to the Clebsch-Gordan problem in quantum mechanics was exploited.
\item With the decoupled equations at hand, we apply the quasi-solution method in order to show the absence of unstable mode solutions. Again, this is significantly more challenging than in \cite{WeiKocDon25} because of the additional parameter $d$.
\end{itemize}

\subsection{Related results}
Geometric wave equations such as the wave maps equation generated a lot of interest in the PDE community in the last few decades. Accordingly, the amount of related literature is vast and impossible to review in its entirety. Thus, we will mostly restrict ourselves to results on blowup for wave maps with spherical targets.

As already mentioned, the stability of $U_\ast$ was studied extensively in corotational symmetry, i.e., for maps $U$ of the form
\begin{align*}
    U(t,x) = \Vector{\sin(u(t,|x|)) \frac{x}{|x|}}{\cos(u(t,|x|))}
\end{align*}
for a scalar function $u$.
In this case, $U$ solves Eq.~\eqref{ex wm eq} if and only if $u$ solves 
\begin{align}\label{corot wm eq}
    \left(\pd_t^2 - \pd_r^2 - \frac{d-1}{r}\pd_r\right)u(t,r) + \frac{(d-1)\sin(u(t,r))\cos(u(t,r))}{r^2} = 0.
\end{align}
Note that $U_\ast$ is corotational with corotational profile 
\begin{align*}
    u_\ast(t,r) = 2 \arctan\left(  \frac{r}{\sqrt{d-2}(1-t)} \right).
\end{align*}
Passing from Eq.~\eqref{ex wm eq}, a system of coupled equations with a nonlinearity that depends on the derivatives of $U$, to Eq.~\eqref{corot wm eq}, a scalar equation with no derivatives in its nonlinearity, is a tremendous simplification.

The nonlinear stability of $u_\ast$ in the class of corotational solutions in backwards lightcones was established for all dimensions $d \geq 3$ in the series of works \cite{Don11,DonSchAic12,CosDonXia16,CosDonGlo17,ChaDonGlo17,DonOst24}. Further stability properties were analyzed such as global-in-space stability \cite{Glo25}, stability in larger space-time regions \cite{DonOst24} and stability in the optimal Sobolev space \cite{DonWal23,DonWal25}, to mention a few.
These stability results all rely crucially on the mode stability of $u_\ast$. Proving the latter consists of solving a spectral problem for a nonself-adjoint Sturm-Liouville operator where standard techniques from mathematical physics fail. Even though numerics heavily suggested that $u_\ast$ is mode stable, see \cite{BizChmTab00,Biz05}, proving this rigorously remained an open problem for quite some time. In fact, already in 2012, the papers \cite{Don11,DonSchAic12} showed that mode stability implies nonlinear stability in $d =3$ but only later, in 2016, mode stability was proved for $d =3$ in \cite{CosDonXia16}.
The result in \cite{CosDonXia16} is based on a sort of brute-force approach and a more systematic method was necessary in order to handle general dimensions $d\geq 3$. 
This was achieved in \cite{CosDonGlo17} with a method developed in Irfan Glogi\'c's PhD thesis \cite{Glo18}. We refer to this method as the \emph{quasi-solution method}. By now, the quasi-solution method was successfully implemented to prove mode stability results also for other equations such as the wave equation with power nonlinearities \cite{GloSch21,CsoGloSch24} and the Yang-Mills equation \cite{CosDonGloHua16}. The proof of mode stability of $U_\ast$ for $d=3$ outside of corotational symmetry in \cite{WeiKocDon25} also relies on this method, combined with novel insight to decouple the equations, see below.

Well before the rigorous proof of stability of $u_\ast$, the paper \cite{BizChmTab00} gave strong numerical evidence of the stability of $u_\ast$ in $d =3$. Moreover, it was observed that whenever a generic solution $u$ of Eq.~\eqref{corot wm eq} experiences finite time blowup, $u$ converges to (a time shifted version of) $u_\ast$. Analogous numerical evidence in higher dimensions $d \geq 4$ is given in \cite{BizBie15}. As a consequence, it is conjectured that $u_\ast$ describes the generic blowup profile of Eq.~\eqref{corot wm eq}. In particular, the blowup seems generically self-similar. On the other hand, in dimensions $d\geq 7$, the existence of another type of blowup (which is not self-similar) is known \cite{GhoIbrNgu18}.

We remark in passing that in the critical case $d =2$ blowup also occurs, albeit with a completely different mechanism. In this critical dimension a result in \cite{Str03} shows that if finite time blowup occurs for a solution $u$ of Eq.~\eqref{corot wm eq} then $u$ must converge, after rescaling, to a non-constant corotational harmonic map. In this case, blowup occurs due to energy concentration at the origin at a rate that is faster than that of self-similar solutions. This in particular excludes self-similar blowup. Numerical evidence for blowup was given in the influential paper \cite{BizChmTab01}. The first rigorous constructions of blowup appeared in \cite{KriSchTat08, RodSte10}. In \cite{RapRod12} blowup solutions with precise asymptotic rates were constructed for each equivariance class and the stability of these solutions in their respective equivariance class was also established. For sufficiently small blowup rates the stability of the solutions constructed in \cite{KriSchTat08} was proved in \cite{KriMia20} in corotational symmetry and remarkably, a stability theory without symmetry assumptions appeared recently in \cite{KriMiaSch24}.

\subsection{Intrinsic formulation}
We will consider the wave maps equation in a chart. For a chart $\Psi: \S^d \supseteq W \to V \subseteq \R^d $, Eq.~\eqref{ex wm eq} transforms into 
\begin{align}\label{int wm eq}
    \pd^\mu \pd_\mu (\Psi \circ U)^k + \Gamma_{i j}^k(\Psi \circ U) \pd^\mu (\Psi \circ U)^i \pd_\mu(\Psi \circ U)^j = 0,\quad k=1,\ldots,d,
\end{align}
where $\Gamma^k_{ij}$ are the associated Christoffel symbols.
We will take $\Psi$ to be the stereographic projection with respect to the south pole. Concretely (for $y = (\Tilde{y},y^{d+1})$)
\begin{align*}
    \Psi(\Tilde{y},y^{d+1}) &= \frac{1}{1 + y^{d+1}} \Tilde{y}\\
    \Psi^{-1}(z) &= \frac{1}{1 + |z|^2} \Vector{2z}{1 - |z|^2},
\end{align*}
which maps $\S^d \setminus \{(0,\ldots,0,-1)\}$ bijectively to $\R^d$.
This choice of a chart is very convenient for us since $F_\ast$ is, up to a scaling factor of $\sqrt{d-2}$, the same as $\Psi^{-1}$.
\subsection{Similarity coordinates}
Since $U_\ast$ is self-similar, it is standard to make the change of variables 
\begin{align*}
    (\tau,\xi) = \chi(t,x) = \left(-\log(1-t) ,\frac{x}{1-t}\right)
\end{align*}
with inverse transformation
\begin{align*}
    (t,x) = \left(1 -  e^{-\tau},  e^{-\tau} \xi\right).
\end{align*}
This transformation maps the truncated backward lightcone with vertex $(1,0)$ to an infinite cylinder with the blowup point shifted to infinity, i.e.,
\begin{align*}
    \chi : \{(t,x) \in [0,1)\times \R^d: |x| \leq 1-t\} \to [0,\infty) \times \overline{\B^d}.
\end{align*}
Then, setting  
\begin{align*}
    v(\tau,\xi) &= (\Psi \circ U)(1 - e^{-\tau}, e^{-\tau}\xi),
\end{align*}
transforms Eq.~\eqref{int wm eq} into
\begin{align}
    0 &= \Big[\partial_\tau^2 + \partial_\tau + 2\xi^j \partial_{\xi^j}\partial_\tau - (\delta^{i j} - \xi^i \xi^j)\pd_{\xi^i}\pd_{\xi^j} + 2 \xi^j \pd_{\xi^j} \Big]v^k(\tau,\xi)\label{int wm eq selfsim}\\
    &+\Gamma_{i j}^k(v(\tau,\xi))\Big[(\partial_\tau + \xi^n\pd_{\xi^n})v^i(\tau,\xi) (\partial_\tau + \xi^n \pd_{\xi^n})v^j(\tau,\xi) - \pd^n v^i(\tau,\xi)\pd_n v^j(\tau,\xi) \Big]   ,\quad k=1,\ldots,d. \notag
\end{align}
Now $v_\ast(\tau,\xi) := v_\ast(\xi) := (\Psi \circ U_\ast)(1- e^{-\tau},  e^{-\tau} \xi)$ is a solution of Eq.~\eqref{int wm eq selfsim} that does not depend on the new time variable $\tau$. In fact, due to our choice of $\Psi$, we simply have
\begin{align*}
    v_\ast(\xi) = \frac{1}{\sqrt{d-2}}\xi.
\end{align*}

\section{Mode stability}
\subsection{Linearization around the solution}
In order to analyze the stability of $v_\ast$, we linearize Eq.~\eqref{int wm eq selfsim} around $v_\ast$. That is to say, we write
\begin{align*}
    v(\tau,\xi) = v_\ast(\xi) + w(\tau,\xi),
\end{align*}
insert this ansatz into Eq.~\eqref{int wm eq selfsim}, and discard all terms of order at least $O(|w|^2)$. This yields
\begin{align*}
    0 &= \Big[\pd_\tau^2 + \pd_\tau + 2 \xi^j\pd_{\xi^j}\pd_\tau - (\delta^{i j} - \xi^i \xi^j)\pd_{\xi^i}\pd_{\xi^j} + 2 \xi^j\pd_{\xi^j}\Big]w^k(\tau,\xi)\\
    &-\frac{4}{d-2 + |\xi|^2}\Big[|\xi|^2 \pd_\tau w^k(\tau,\xi) - (1 - |\xi|^2)\xi^j \pd_{\xi^j} w^k(\tau,\xi) + \xi^k \pd_j w^j(\tau,\xi) - \xi_j \pd^k w^j(\tau,\xi)\Big]\\
    &-\frac{2(d-2- |\xi|^2)}{d-2 + |\xi|^2}w^k(\tau,\xi), \quad k=1,\ldots,d.
\end{align*}
The following transformation is useful (and brings the equation into the same form as in \cite{CosDonGlo17})
\begin{align*}
    w(\tau,\xi) = (d-2+  |\xi|^2) \phi(\tau,\xi),
\end{align*}
which gives the equations
\begin{align}
    0 &= \Big[\partial_\tau^2 + \partial_\tau + 2\xi^j \partial_{\xi^j}\partial_\tau - (\delta^{i j} - \xi^i \xi^j)\pd_{\xi^i}\pd_{\xi^j} + 2 \xi^j \pd_{\xi^j} \Big]\phi^k(\tau,\xi)\label{lin eq}\\
    &- \frac{4(d-1)(d-2 - |\xi|^2)}{(d-2 + |\xi|^2)^2}\phi^k(\tau,\xi) - \frac{4}{d-2  +|\xi|^2}(\xi^k \pd_{\xi^j}  - \xi_j \pd_{\xi_k} )\phi^j(\tau,\xi), \quad k=1,\ldots,d \notag.
\end{align}

\subsection{Mode solutions}
Now we look for \emph{mode solutions} of Eq.~\eqref{lin eq}. These have the form
\begin{align}\label{mode ansatz}
    \phi(\tau,\xi) = e^{\lambda \tau} f(\xi)
\end{align}
for some $\lambda \in \C$ and $f \in C^\infty(\overline{\B^d},\C^d)$. Assuming that $f$ is smooth is not a real restriction. This technically depends on the space one is working in, but since the stability analysis will be carried out in some Sobolev space with sufficiently high regularity, one can show with an elliptic regularity argument that solutions of Eq.~\eqref{lin eq} in the sense of distributions that have a certain regularity must already be smooth. See the discussion in \cite{Don24} for more details.

Inserting \eqref{mode ansatz} into Eq.~\eqref{lin eq} yields the equations
\begin{align}
    0 &= \left[ (\delta^{i j} - \xi^i \xi^j)\pd_{\xi^i }\pd_{\xi^j}  - 2 (\lambda + 1 )\xi^j \pd_{\xi^j} - \lambda(\lambda + 1) + \frac{4(d-1)(d-2 - |\xi|^2)}{(d-2 + |\xi|^2)^2} \right]f^k(\xi)\label{full mode eq}\\
    &+ \frac{4}{d-2 + |\xi|^2} (\xi^k \pd_{\xi^j} - \xi_j \pd_{\xi_k}) f^j(\xi), \quad k = 1 ,\ldots,d\notag.
\end{align}
Heuristically, one expects instability of $v_\ast$ if there are mode solutions with $\Real \lambda > 0$. Our goal will be to rule out such solutions. However, due to the continuous symmetries of Eq.~\eqref{ex wm eq}, there actually exist mode solutions with $\Real \lambda \geq 0$. If $U$ solves \eqref{ex wm eq} then so do the following functions.
\begin{itemize}
    \item \textbf{Translations in time and space:} $U_{(T,X)}(t,x) = U(t-T,x-X )$ for $(T,X) \in \R^{1,d}$.
    \item \textbf{Scaling:} $U_a(t,x) = U\left( \frac{t}{a},\frac{x}{a}  \right)$ for $a > 0$.
    \item \textbf{Rotations on the domain:} $U_R(t,x) = U(t,R x)$ for $R \in SO(d)$.
    \item \textbf{Lorentz boosts:} $U_{L_j(a)}(t,x) = U(L_j(a)(t,x))$ for $j \in \{1 , \ldots,d\}$ and $a \in \R$, where $L_j(a)$ denotes the Lorentz boost of rapidity $a$ in direction $e_j$.
    \item \textbf{Rotations on the target:} $U^R(t,x) = R U(t,x)$ for $R \in SO(d+1)$.
\end{itemize}
In principle, each of these symmetries generates a mode solution. However, there are some redundancies.
\begin{itemize}
    \item Scaling a self-similar solution is essentially the same as translating in time
    \begin{align*}
        (U_\ast)_a(t,x) = U_\ast\left( \frac{t}{a},\frac{x}{a} \right) = F_\ast\left( \frac{\frac{x}{a}}{1 - \frac{t}{a}} \right) = F_\ast\left( \frac{x}{a - t} \right) = U_\ast(t + 1-a,x).
    \end{align*}
    \item Assume that $R$ is a rotation on $\R^{d+1}$ that leaves the last component invariant, i.e.,
    \begin{align*}
        R \Vector{\Tilde{y}}{y^{d+1}} = \Vector{\Tilde{R} \Tilde{y}}{y^{d+1}}
    \end{align*}
    for some rotation $\Tilde{R}$ on $\R^d$. Then one checks that 
    \begin{align*}
        U^R_\ast(t,x)=R U_\ast(t,x) = R F_\ast\left(\frac{x}{1-t}\right) = F_\ast \left(\Tilde{R} \frac{x}{1-t}\right) = U_\ast(t,\Tilde{R} x).
    \end{align*}
\end{itemize}
Counting the dimensions of the corresponding symmetry groups, we expect 
\begin{align*}
    d+1 + \binom{d}{2} + d  + d = 3d+ 1 + \binom{d}{2}
\end{align*}
many linearly independent symmetry mode solutions. Indeed, these are explicitly given (without the weight $\frac{1}{d- 2 + |\xi|^2}$) by
\begin{align*}
    e^\tau  \xi&\\
    e^\tau e_j&, \quad 1 \leq j \leq d\\
    \xi_j e_k - \xi_k e_j&, \quad 1 \leq j < k \leq d\\
    |\xi|^2 e_j - d\xi_j \xi &,\quad 1 \leq j \leq d\\
    (d-  |\xi|^2) e_j&,\quad 1 \leq j \leq d,
\end{align*}
where $e_j$ denotes the $j$-th unit vector in $\C^d$. These ``instabilities'' reflect the fact that one cannot expect the single solution $v_\ast$ to be stable but one has to consider the whole family of solutions that is obtained by applying the above symmetries to $v_\ast$. In this sense, the above mode solutions, henceforth called \emph{symmetry modes}, are ``artificial'' and do not indicate the existence of ``real'' instabilities. This observation motivates the following definition.
\begin{defin}
We say that $U_\ast$ is \emph{mode stable} if the only nontrivial $f \in C^\infty(\overline{\B^d},\C^d)$ that solve Eq.~\eqref{full mode eq} with $\Real \lambda \geq 0$ are linear combinations of the symmetry modes.
\end{defin}
With these preparations we can now state the main result of the present paper.
\begin{thm}\label{mode stability}
The solution $U_\ast$ is mode stable.
\end{thm}
\subsection{Decoupling of the equations}
A quick inspection of Eq.~\eqref{full mode eq} reveals that the $d$ equations are coupled due to the last term. Ultimately, our goal is to employ a similar approach as in \cite{GloSch21,CsoGloSch24}, namely decomposing into spherical harmonics. In order to do this, we first analyze the term responsible for the coupling. We define
\begin{align*}
    (K_{k j})_\xi := \xi_k \pd_{\xi^j} - \xi_j \pd_{\xi^k}
\end{align*}
which acts on functions with values in $\C^d$. Note that $K$ is an angular differential operator in the sense that it maps radial functions to 0. Consider for $\ell \geq 0$ the following space (commonly referred to as the \emph{$\ell$-th spherical harmonics})
\begin{align*}
    \Y_\ell := \{ p \in \C[\xi] : \Delta p = 0  \text{ and }p\text{ is homogeneous of degree }\ell\}
\end{align*}
and denote by $\Y_\ell^d$ the space of $\C^d$-valued functions where each component is in $\Y_\ell$.

Now recall that if one considers the restriction of the polynomials to $\S^{d-1}$, one has the orthogonal decomposition 
\begin{align*}
    L^2(\S^{d-1}) = \bigoplus_{\ell \geq 0} \mathbb{Y}_\ell,
\end{align*}
see e.g.~\cite{AtkHan12}, where the inner product is given by
\begin{align*}
    (f|g)_{L^2(\S^{d-1})} = \int_{\S^{d-1}} f(\omega) \overline{g(\omega)}d \sigma(\omega),
\end{align*}
with $\sigma$ the surface measure on $\S^{d-1}$.

One checks that $K$ leaves the spaces $\Y_\ell^d$ invariant and one can use methods from the theory of Lie algebras to show that $K$ is diagonalizable on this space with eigenvalues
\begin{align*}
    -\ell,\quad 1,\quad \ell+d-2
\end{align*}
for $\ell \geq 1$, see \autoref{Decoupling of the equations} for details. The case $\ell=0$ is trivial since $K$ maps constant functions to 0. We denote the eigenspace of $K|_{\Y_\ell^d}$ corresponding to the eigenvalue $m$ by $W_{\ell,m}$. This yields the decomposition
\begin{align}\label{decomp of vector L2}
    L^2(\S^{d-1},\C^d) = \Y_0^d \oplus \bigoplus_{\ell \geq 1} \bigoplus_{m \in \{-\ell,1,\ell+d-2\}} W_{\ell,m}.
\end{align}
Note that each entry $K_{k j}$ of $K$ is skew-symmetric on $\Y_\ell$ with respect to $L^2(\S^{d-1})$. Since $K_{k j} = - K_{j k}$, one readily verifies that this implies that $K$ is symmetric on $\Y_\ell^d$ with respect to $L^2(\S^{d-1},\C^d)$, where the inner product is given by
\begin{align*}
    (f|g)_{L^2(\S^{d-1})} = \int_{\S^{d-1}} f(\omega) \cdot \overline{g(\omega)}d \sigma(\omega) = \int_{\S^{d-1}} f^j(\omega) \overline{g_j(\omega)}d \sigma(\omega).
\end{align*}
Hence the $W_{\ell,m}$ for $m \in \{-\ell,1,\ell+d-2\}$ are mutually orthogonal and, since $\Y_\ell^d \perp \Y_{\ell'}^d$ for $\ell \not=\ell'$, the decomposition \eqref{decomp of vector L2} is orthogonal. 
Now we choose $\{Y_{0,\alpha}\}_\alpha$ an orthonormal basis of $\Y_0^d$ and $\{Y_{\ell,m,\alpha}\}_\alpha$ an orthonormal basis of $W_{\ell,m}$. Then for $f \in C^\infty(\overline{\B^d},\C^d)$ we can decompose its ``angular part''
\begin{align}\label{sph harm decomp}
    f(\rho \cdot ) = \sum_{\alpha} f_{0,\alpha}(\rho) Y_{0,\alpha} + \sum_{\ell \geq 1} \sum_{m \in \{-\ell,1,\ell+d-2\}} \sum_{\alpha} f_{\ell,m,\alpha}(\rho) Y_{\ell,m,\alpha}.
\end{align}
The functions $f_{\ell,m,\alpha}$ are just coefficients of $L^2(\S^{d-1},\C^d)$-projections
\begin{align*}
    f_{\ell,m,\alpha}(\rho) = (f(\rho\cdot )|Y_{\ell,m,\alpha})_{L^2(\S^{d-1},\C^d)} = \int_{\S^{d-1}} f(\rho \omega)\cdot \overline{Y_{\ell,m,\alpha}(\omega)} d \sigma(\omega),
\end{align*}
analogously for $f_{0,\alpha}$.
One verifies with this that $f_{0,\alpha},f_{\ell,m,\alpha}:[0,1] \to \C$ are smooth.

Now recall that $\Y_\ell$ (and hence also $\Y_\ell^d$) is the eigenspace of the Laplace-Beltrami operator $\Delta^{\S^{d-1}}$ corresponding to the eigenvalue $- \ell(\ell+d-2)$. Combining this with the fact that $Y_{0,\alpha},Y_{\ell,m,\alpha}$ are eigenfunctions of $K$, we can insert \eqref{sph harm decomp} into Eq.~\eqref{full mode eq} to obtain the equations
\begin{align}\label{mode ode zero}
    f_{0,\alpha}''(\rho) + \frac{d-1 - 2 (\lambda + 1)\rho^2}{\rho(1 - \rho^2)}f_{0,\alpha}'(\rho) + \left[-\frac{\lambda(\lambda + 1)}{1 - \rho^2} + V_0(\rho)\right]f_{0,\alpha}(\rho) = 0
\end{align}
and
\begin{align}\label{mode ode l m}
    f_{\ell,m,\alpha}''(\rho) + \frac{d-1 - 2 (\lambda + 1)\rho^2}{\rho(1 - \rho^2)}f_{\ell,m,\alpha}'(\rho) + \left[-\frac{\lambda(\lambda + 1)}{1 - \rho^2} + V_{\ell,m}(\rho)\right]f_{\ell,m,\alpha}(\rho) = 0
\end{align}
where
\begin{align*}
    V_0(\rho) = \frac{4(d-1)(d-2 - \rho^2)}{(1 - \rho^2)(d- 2 + \rho^2)^2}
\end{align*}
and
\begin{align*}
    V_{\ell,m}(\rho) &= \frac{-(d-2)^2 \ell(\ell+d-2) + 2 (d-2)[-2 +2d + 2m - \ell(\ell+d-2)]\rho^2 }{\rho^2(1 - \rho^2)(d-2 + \rho^2)^2}\\
    &+\frac{[4 -4d + 4m - \ell(\ell+d-2)]\rho^4}{\rho^2(1 - \rho^2)(d-2 + \rho^2)^2}.
\end{align*}
Since $\alpha$ does not appear in the equation, we omit it from now on.

The symmetry modes now give rise to the smooth solutions (which we will again call symmetry modes)
\begin{align*}
    \begin{cases}
        f_{0}^0(\rho) = \frac{d-\rho^2}{d-2 + \rho^2},& \ell=0,\lambda=0\\
        f_0^1(\rho) = \frac{1}{d-2 + \rho^2},& \ell=0,\lambda=1\\
        f_{1,1+d-2}^1(\rho) = \frac{\rho}{d-2 + \rho^2},& \ell=1,m=1+d-2,\lambda =1\\
        f_{1,1}^0(\rho)=\frac{\rho}{d-2 + \rho^2},& \ell=1,m=1,\lambda=0\\
        f_{2,2+d-2}^0(\rho) =\frac{ \rho^2}{d-2 + \rho^2},& \ell=2,m=2+d-2,\lambda =0
    \end{cases}.
\end{align*}
The corotational case is $\ell=1$ and $m = 1 +d -2 = d-1$ in this notation.
Our goal now is to prove that these are in fact the only nontrivial smooth solutions of Eqs.~\eqref{mode ode zero} and \eqref{mode ode l m}.

It should be noted that excluding smooth solutions of equations such as \eqref{mode ode zero} and \eqref{mode ode l m} constitutes a hard spectral problem, since this is highly nonself-adjoint. This is due to the fact that $\lambda$ appears as a coefficient of the first order derivative term. One can remove the first order term and relate the resulting equation to an eigenvalue equation of a Sturm-Liouville problem on a weighted space $L_w^2(0,1)$ with $w(\rho) = \frac{1}{(1 - \rho^2)^2}$. Here standard methods can be applied for excluding the existence of smooth solutions. The problem is that the transformation itself depends on $\lambda$. In particular, if $\Real \lambda \leq \frac{d-3}{2} + 1$, then a function $f \in C^\infty([0,1])$ is not necessarily mapped into the space $L^2_w(0,1)$. See the discussion in \cite{Don24} for more on this.
\subsection{Supersymmetric removal}
Before proceeding, we have to ``remove'' the symmetry modes for the cases where they occur, namely twice for $\ell=0$ and once for $(\ell,m) \in \{(1,1+d-2),(1,1),(2,2+d-2)\}$. We achieve this using a factorization procedure coming from supersymmetric quantum mechanics.

We explain this for (one of the two solutions in) $\ell=0$. Take Eq.~\eqref{mode ode zero} (dropping the subscripts)
\begin{align*}
    f''(\rho) + \frac{d-1 - 2 (\lambda + 1)\rho^2}{\rho(1 - \rho^2)} f'(\rho) + \left[-\frac{\lambda(\lambda + 1)}{1 - \rho^2} + V_0(\rho)\right]f(\rho) = 0
\end{align*}
and transform it via
\begin{align*}
    f(\rho) = \rho^{\frac{1 - d}{2}} (1 - \rho^2)^{\frac{d-3 -2 \lambda}{4}} g(\rho).
\end{align*}
This yields the equation
\begin{align}\label{first order removed}
    g''(\rho) + \left[V_0(\rho) + \frac{(3-d )(d-1 + 2 \rho^2)}{4 \rho^2(1 - \rho^2)^2}\right] g(\rho) = \frac{\lambda(\lambda - (d-1))}{(1 - \rho^2)^2}g(\rho).
\end{align}
The symmetry mode yields the function
\begin{align*}
    g_{0}^0(\rho) &= \frac{\rho^\frac{d-1}{2}(1 - \rho^2)^\frac{3-d}{4}(d-\rho^2)}{d- 2 + \rho^2}
\end{align*}
which, by construction, satisfies 
\begin{align*}
    (g_0^0)''(\rho) + \left[V_0(\rho) + \frac{(3-d )(d-1 + 2 \rho^2)}{4 \rho^2(1 - \rho^2)^2}\right] g_0^0(\rho) =0,
\end{align*}
i.e.,
\begin{align*}
    -\frac{(g_0^0)''(\rho)}{g_0^0(\rho)} = V_0(\rho) + \frac{(3-d )(d-1 + 2 \rho^2)}{4 \rho^2(1 - \rho^2)^2}.
\end{align*}
From this one can calculate that 
\begin{align*}
    \pd_\rho^2 + V_0(\rho) + \frac{(3-d )(d-1 + 2 \rho^2)}{4 \rho^2(1 - \rho^2)^2} = \left(\pd_\rho + \frac{(g_0^0)'(\rho)}{g_0^0(\rho)}\right)\left(\pd_\rho - \frac{(g_0^0)'(\rho)}{g_0^0(\rho)}\right)
\end{align*}
and hence we can rewrite Eq.~\eqref{first order removed} as 
\begin{align*}
    (1 - \rho^2)^2 \left(\pd_\rho + \frac{(g_0^0)'(\rho)}{g_0^0(\rho)}\right)\left(\pd_\rho - \frac{(g_0^0)'(\rho)}{g_0^0(\rho)}\right) g(\rho) = \lambda(\lambda - (d-1))g(\rho).
\end{align*}
We now apply $\left(\pd_\rho - \frac{(g_0^0)'(\rho)}{g_0^0(\rho)}\right)$ to this equation and define $h(\rho )= \left(\pd_\rho - \frac{(g_0^0)'(\rho)}{g_0^0(\rho)}\right)g(\rho)$. This yields the equation
\begin{align}\label{first removed 0 eq}
    h''(\rho) - \frac{4 \rho}{1 - \rho^2}h'(\rho) + V(\rho) h(\rho) = \frac{\lambda(\lambda - (d-1))}{(1 - \rho^2)^2}h(\rho),
\end{align}
where now
\begin{align*}
    V(\rho) &= - V_0(\rho) - \frac{(3-d )(d-1 + 2 \rho^2)}{4 \rho^2(1 - \rho^2)^2} - 2 \left(\frac{(g_0^0)'(\rho)}{g_0^0(\rho)}\right)^2 - \frac{4 \rho}{1 - \rho^2}\frac{(g_0^0)'(\rho)}{g_0^0(\rho)}.
\end{align*}
The point of this is that, by construction, $\left(\pd_\rho  - \frac{(g_0^0)'(\rho)}{g_0^0(\rho)} \right)g_0^0(\rho) = 0$, so the symmetry mode no longer induces a nontrivial solution of Eq.~\eqref{first removed 0 eq}.

Now one has to repeat this procedure for this equation for $\lambda =1$ and for the other three equations. Finally one brings the equation into a similar form as Eqs.~\eqref{mode ode zero} and \eqref{mode ode l m}. More precisely, one obtains the equation
\begin{align}\label{susy removed eq}
    \Tilde{f}''(\rho) + \frac{d-1 - 2 (\lambda + 1)\rho^2}{\rho(1 - \rho^2)}\Tilde{f}'(\rho) + \left[ -\frac{\lambda(\lambda + 1)}{1 - \rho^2} + \Tilde{V}_{\ell,m}(\rho) \right]\Tilde{f}(\rho) = 0,
\end{align}
where
\begin{align}\label{transformed V}
    \Tilde{V}_{\ell,m}(\rho) = \begin{cases}
        -\frac{2d}{\rho^2(1 - \rho^2)}, &\ell=0\\
        -\frac{2(d-2)(d-\rho^2)}{\rho^2(1 - \rho^2)(d-2 + \rho^2)},&\ell=1,m=1+d-2\\
        \frac{-2d(d-2) +2 \rho^4}{\rho^2(1 - \rho^2)(d-2 + \rho^2)}, &\ell=1,m=1\\
        \frac{-3(d+1)(d-2) + (d-3)\rho^2}{\rho^2(1 - \rho^2)(d- 2 + \rho^2)},&\ell=2,m=2+d-2
    \end{cases}.
\end{align}
The transformation is given by 
\begin{align*}
    &\Tilde{f}(\rho)\\
    &=\rho^{-\frac{d-1}{2}} (1 - \rho^2)^\frac{d+1 - 2 \lambda}{4} \left( \pd_\rho  - \frac{(\Tilde{g}_0^1)'(\rho)}{\Tilde{g}_0^1(\rho)} \right)(1 - \rho^2) \left( \pd_\rho - \frac{(g_0^0)'(\rho)}{g_0^0(\rho)}  \right) \rho^\frac{d-1}{2}(1-\rho^2)^\frac{2\lambda - (d-3)}{4}f(\rho),
\end{align*}
for $\ell=0$ and
\begin{align*}
    \Tilde{f}(\rho) = \begin{cases}
        \rho^{-\frac{d-1}{2}}(1 - \rho^2)^\frac{d+1 - 2 \lambda}{4} \left( \pd_\rho - \frac{g_{1,1+d-2}'(\rho)}{g_{1,1+d-2}(\rho)} \right) \rho^\frac{d-1}{2}(1 - \rho^2)^\frac{2 \lambda - (d-3)}{4} f(\rho), & \ell=1,m=1+d-2\\
        \rho^{-\frac{d-1}{2}}(1 - \rho^2)^{\frac{d+1 - 2 \lambda}{4}}\left( \pd_\rho - \frac{g_{1,1}'(\rho)}{g_{1,1}(\rho)} \right) \rho^\frac{d-1}{2}(1 - \rho^2)^\frac{2\lambda - (d-3)}{4} f(\rho), & \ell=1,m=1\\
        \rho^{-\frac{d-1}{2}}(1 - \rho^2)^{\frac{d+1 - 2 \lambda}{4}}\left( \pd_\rho - \frac{g_{2,2+d-2}'(\rho)}{g_{2,2+d-2}(\rho)} \right) \rho^\frac{d-1}{2}(1 - \rho^2)^\frac{2\lambda - (d-3)}{4} f(\rho), & \ell=2,m=2+d-2
    \end{cases},
\end{align*}
where
\begin{align*}
    \frac{(g_0^0)'(\rho)}{g_0^0(\rho)}&= \frac{d(d-1)(d-2) - 2 (d^2 + d - 3)\rho^2 + (7d - 11)\rho^4 + 2 \rho^6}{2 \rho(1 - \rho^2)(d- \rho^2)(d-2 + \rho^2)} \\
    \frac{(\Tilde{g}_0^1)'(\rho)}{\Tilde{g}_0^1(\rho)} &= \frac{d(d+1)  + (3 - 7d)\rho^2 + 2 \rho^4}{2 \rho(1 - \rho^2)(d-\rho^2)}\\
    \frac{g_{1,1+d-2}'(\rho)}{g_{1,1+d-2}(\rho)} &= \frac{(d+1)(d-2) + (- 5 d + 9)\rho^2  - 2 \rho^4}{2 \rho(1 - \rho^2)(d-2 + \rho^2)}\\
    \frac{g_{1,1}'(\rho)}{g_{1,1}(\rho)} &=\frac{(d+1 )(d-2) + (5 - 3d)\rho^2 }{2 \rho(1- \rho^2)(d-2 + \rho^2)}\\
    \frac{g_{2,2+d-2}'(\rho)}{g_{2,2+d-2}(\rho)} &=\frac{(d-2)(d+3) + (11 - 5d )\rho^2 - 2 \rho^4 }{2 \rho (1 - \rho^2)(d-2 + \rho^2)}.
\end{align*}
One then has to check that this transformation preserves smoothness.
\begin{lem}
Let either $\ell=0,\lambda \in \C\setminus \{0,1\}$, $\ell=1,m=1+d-2,\lambda \in \C \setminus\{1\}$, $\ell=1,m=1,\lambda \in \C \setminus\{0\}$ or $\ell=2,m=2+d-2,\lambda \in \C\setminus\{0\}$ and assume that $f \in C^\infty([0,1])$ is a nontrivial solution of Eq.~\eqref{mode ode zero}, respectively Eq.~\eqref{mode ode l m}. Then $\Tilde{f}$ defined by the above transformation is a nontrivial solution of Eq.~\eqref{susy removed eq} and $\Tilde{f} \in C^\infty([0,1])$.
\end{lem}
\begin{proof}
By construction $\Tilde{f}$ solves Eq.~\eqref{susy removed eq}. Then one checks in each case that if $f \not=0$ is mapped to $0$ by this transformation, then either $f$ does not solve the original equation, or, if it does, then it is not smooth.

Concerning smoothness of $\Tilde{f}$, we see that smoothness on $(0,1)$ is preserved.

We treat the two endpoints separately using Frobenius theory, see for example \cite{Tes12}.
\begin{itemize}
    \item $\rho=0:$ The equations \eqref{mode ode zero},\eqref{mode ode l m} have the Frobenius indices $\{\ell,-(\ell+d-2)\}$ at $\rho=0$. Since $-(\ell+d-2) < 0$ we must have $f(\rho) = \rho^\ell h(\rho)$ for a function $h$ that is analytic around $\rho=0$ with $h(0)\not=0$. Hence we can write 
    \begin{align*}
        \rho^\frac{d-1}{2}(1 - \rho^2)^\frac{2\lambda - (d-3)}{4}f(\rho) = \rho^{\ell+  \frac{d-1}{2}}\Tilde{h}(\rho),
    \end{align*}
    where again $\Tilde{h}$ is analytic around $\rho=0$. Then for any $w$ analytic around $\rho=0$, one calculates
    \begin{align*}
        &\rho^{-\frac{d-1}{2}} (1 - \rho^2)^\frac{d+1 - 2 \lambda}{4} \left(\pd_\rho - \frac{w(\rho)}{\rho}\right) [\rho^{\ell + \frac{d-1}{2}} \Tilde{h}(\rho)]\\
        &= \rho^{\ell-1}(1 - \rho^2)^\frac{d+1 - 2 \lambda}{4} \left[\left(\ell + \frac{d-1}{2}\right) \Tilde{h}(\rho) + \rho \Tilde{h}'(\rho) - w(\rho) \Tilde{h}(\rho) \right],
    \end{align*}
    which certainly is again analytic around $\rho=0$ for $\ell \geq 1$.
    
    For $\ell=0$ one checks similarly that for $f$ analytic around $\rho=0$, one can find $h$ analytic around $\rho=0$ such that 
    \begin{align*}
        \Tilde{f}(\rho) = \rho^{-2} h(\rho).
    \end{align*}
    In particular $|\Tilde{f}(\rho)| \lesssim \rho^{-2}$ near $\rho=0$. Since $\Tilde{f}$ solves Eq.~\eqref{susy removed eq} which has the Frobenius indices $\{2,-d\}$ at $\rho=0$, we can exclude that it corresponds to the order $-d$ and hence must be of the form $\Tilde{f}(\rho) = \rho^2 \Tilde{h}(\rho)$ with $\Tilde{h}$ analytic around $\rho=0$. This means in particular that $\Tilde{f}$ is analytic and hence smooth around $\rho=0$.
    \item $\rho=1:$ The equations \eqref{mode ode zero}, \eqref{mode ode l m} are both of Fuchsian type. This means a solution $f$, which is smooth around $\rho=1$, is already analytic around $\rho=1$. Then a calculation shows that again $\Tilde{f}$ is analytic around $\rho=1$ and in particular smooth.
\end{itemize}
\end{proof}
\subsection{Standard Heun form}
We set $\Tilde{V}_{\ell,m} = V_{\ell,m}$ except in the four special cases in \eqref{transformed V}. Observe that Eq.~\eqref{susy removed eq} is of Heun type and has the singularities $0,\pm 1,\pm \sqrt{d-2} i,\infty$. We set $\Tilde{f}(\rho) = g(\rho^2)$ and $x = \rho^2$. This transforms Eq.~\eqref{susy removed eq} into 
\begin{align*}
    g''(x) + \frac{d - (2\lambda + 3)x}{2x(1-x)}g'(x) + \left[ - \frac{\lambda(\lambda + 1)}{4x(1-x)} + \frac{\Tilde{V}_{\ell,m}(\sqrt{x})}{4x} \right]g(x) = 0.
\end{align*}
This equation now only has the four singular points $0,1,-(d-2),\infty$. Then bringing this into standard Heun form yields
\begin{align}\label{standard Heun eq}
    h''(x) + p_{\ell,m}(x)h'(x) + q_{\ell,m}(x) h(x) =0,
\end{align}
where
\begin{align*}
    p_{\ell,m}(x) = \begin{cases}
        \frac{d+4}{2x} + \frac{2\lambda + 3-d }{2(x-1)}, &\ell=0\\
        \frac{d+4}{2x} + \frac{2\lambda  + 3-d }{2(x-1)}, &\ell=1,m=1+d-2\\
        \frac{d+4}{2x} + \frac{2\lambda + 3-d }{2(x-1)}, &\ell=1,m=1\\
        \frac{d+6}{2x} + \frac{2\lambda  + 3-d }{2(x-1)},&\ell=2,m=2+d-2\\
        \frac{d+2\ell}{2x} + \frac{2\lambda  +3-d }{2(x-1)} - \frac{2}{d-2+ x},&\text{else}
    \end{cases}
\end{align*}
and
\begin{align*}
    q_{\ell,m}(x) = \begin{cases}
        \frac{(\lambda + 2)(\lambda + 3)}{4x(x-1)},& \ell=0\\
        \frac{(d-2)\lambda^2 + 5(d-2)\lambda +2d -8 + (\lambda + 2)(\lambda +3)x}{4x(x-1)(d-2 + x)}, & \ell=1,m=1+d-2\\
        \frac{(d-2)\lambda^2 + 5(d-2)\lambda + 4d -12 + (\lambda + 1)(\lambda +4)x}{4x(x-1)(d-2+ x)},&\ell=1,m=1\\
        \frac{(d-2) \lambda^2 + 7 (d-2) \lambda + 8 d -24 + (\lambda + 3)(\lambda + 4)x}{4x(x-1)(d-2+ x)},&\ell=2,m=2-d+2\\
        \frac{(d-2)\ell^2 + (d+2)\ell + (d-2)\lambda (2\ell + \lambda + 1) + 4 -2d - 4m + (\lambda + \ell-1)(\lambda + \ell -2)x}{4x(x-1)(d-2 + x)}, &\text{else}
    \end{cases}.
\end{align*}
The transformation is given by 
\begin{align*}
    g(x) = \begin{cases}
        x h(x),&\ell=0\\
        x h(x), &\ell=1,m=1+d-2\\
        x h(x),& \ell=1,m=1\\
        x^\frac{3}{2}h(x),&\ell=2,m=2+d-2\\
        x^\frac{\ell}{2}(x+d-2)^{-1} h(x),&\text{else}
    \end{cases}.
\end{align*}
Clearly, $h$ defined in this way is again a nontrivial solution for $\rho^2=x \in (0,1]$. It is also smooth around $x =0$, which is checked using Frobenius analysis and a parity argument.
\subsection{\texorpdfstring{$\ell=0$}{l = 0}}
\label{l0 case}
For $\ell=0$ one can rewrite Eq.~\eqref{standard Heun eq} as
\begin{align*}
    x(1-x)h''(x) + \left(c - \left(a+b + 1\right)x \right) h'(x) - ab h(x) = 0
\end{align*}
where
\begin{align*}
    a &= \frac{\lambda +2}{2}\\
    b &= \frac{\lambda + 3}{2}\\
    c &= \frac{d+4}{2}.
\end{align*}
This is a hypergeometric differential equation (see e.g. \cite{NIST:DLMF}) and its smooth solution around $x =0$ is explicitly given as
\begin{align*}
    h(x) = {_2 F_1}(a,b,c;x) = \sum_{n=0}^\infty \frac{(a)_n (b)_n}{(c)_n n!} x^n,
\end{align*}
where $(a)_n := a(a+1)\cdots (a+n-1)$ is the Pochhammer symbol. None of $a,b,c$ can be nonpositive integers for $\Real \lambda \geq 0$ and hence the ratio test shows that the radius of convergence with respect to $x =0$ is $1$. Since the hypergeometric equation only has the singular points $x =0,1,\infty$, we conclude that ${_2F_1}(a,b,c;\cdot)$ cannot be smooth at $x =1$ and in particular does not belong to $C^\infty([0,1])$.
\subsection{Recurrence relation}
For $\ell \geq 1$, we are dealing with genuine Heun equations, where there are no explicit solution formulae available. In particular, the 
\emph{connection problem} is open. As a consequence, it is unknown how a solution that is smooth near $x=0$ relates to a solution that is smooth near $x=1$.

Nevertheless, we attempt to analyze the behavior of the coefficients of the power series. Inserting the ansatz
\begin{align*}
    h(x) = \sum_{n = 0}^\infty a_n x^n
\end{align*}
into Eq.~\eqref{standard Heun eq} yields the recurrence relation
\begin{align*}
    a_{n+2} = A_n(\lambda) a_{n+1} +  B_n(\lambda) a_n
\end{align*}
for all $n \geq -1$, where we set $a_{-1}= 0$ for convenience, with
\begin{align}\label{An special cases}
    A_n(\lambda) = \begin{cases}
        \frac{4(d-3)n^2 + [ 4(d-2)\lambda + 16(d-3) ]n + (d-2)\lambda^2 + 9(d-2)\lambda + 14d - 44}{2(d-2)(n+2)(2n + d +6)}, &\ell=1,m=1+d-2\\
        \frac{4(d-3)n^2 + [ 4(d-2)\lambda + 16(d-3) ]n + (d-2)\lambda^2 + 9(d-2)\lambda + 16(d - 3)}{2(d-2)(n+2)(2n + d +6)}, & \ell=1,m=1\\
         \frac{4(d-3)n^2 + [4(d-2)\lambda + 20 (d-3)]n + (d-2) \lambda^2 + 11(d-2)\lambda + 24(d -3)}{2(d-2)(n+2)(2n + d +8)}, & \ell=2,m=2+d-2
    \end{cases}
\end{align}
and otherwise
\begin{align}\label{An general case}
    A_{n,d,\ell,m}(\lambda):=A_n(\lambda) &:=\frac{4 (d-3)n^2 + [4(d-2)\lambda +4(d-3)\ell + 8d -16 ]n}{2(d-2)(n+2)(2n + 2\ell + d +2)}\\
        &+\frac{(d-2)\lambda^2 + (d-2)(2\ell + 5)\lambda + (d-2)\ell^2 + 5(d-2)\ell + 2d -4m}{2(d-2)(n+2)(2n + 2\ell + d +2)}\notag
\end{align}
as well as
\begin{align}\label{Bn all cases}
    B_{n,d,\ell,m}(\lambda):=B_n(\lambda) &:= \begin{cases}
        \frac{(2n+  \lambda + 2)(2n + \lambda + 3)}{2(d-2)(n+2)(2n + d +6)}, &\ell=1,m=1+d-2\\
        \frac{(2n+  \lambda + 1)(2n + \lambda + 4)}{2(d-2)(n+2)(2n + d +6)}, & \ell=1,m=1\\
         \frac{(2n + \lambda + 3)(2n + \lambda + 4)}{2(d-2)(n+2)(2n + d +8)}, & \ell=2,m=2+d-2\\
         \frac{(2n + \lambda + \ell - 2)(2n + \lambda + \ell -1)}{2(d-2)(n+2)(2n + 2\ell + d +2)}, & \text{else}
    \end{cases}.
\end{align}
We will mostly suppress the parameters $d,\ell,m$ in the subscripts for brevity.
\begin{defin}
For $\lambda \in \C$, the sequence $(a_n(\lambda))_{n \geq -1}$ is defined recursively by $a_{-1}(\lambda) =0$, $a_0(\lambda) = 1$ and 
\begin{align}\label{a recurr rel}
    a_{n+2}(\lambda) = A_n(\lambda) a_{n+1}(\lambda) + B_n(\lambda) a_n(\lambda)
\end{align}
for $n \geq -1$.
\end{defin}
Our ultimate goal is to prove that the radius of convergence of $h$ cannot be too large. Hence we will analyze the ratios of $a_n(\lambda)$. We actually have the following dichotomy.
\begin{lem}\label{ratios converge}
Let $\Real \lambda \geq 0$. Then $\frac{a_{n+1}(\lambda)}{a_n(\lambda)}$ converges as $n \to \infty$ and
\begin{align*}
    \lim_{n \to \infty} \frac{a_{n+1}(\lambda)}{a_n(\lambda)} \in \left\{1, - \frac{1}{d-2}  \right\}.
\end{align*}
\end{lem}
\begin{proof}
We want to apply Poincaré's theorem on difference equations \autoref{poincare}, so we calculate
\begin{align*}
    \lim_{n \to \infty}A_n(\lambda) &= \frac{d-3}{d-2}\\
    \lim_{n \to \infty}B_n(\lambda) &= \frac{1}{d-2}.
\end{align*}
Hence the characteristic equation of the recursion \eqref{a recurr rel} is 
\begin{align*}
    \alpha^2 - \frac{d-3}{d-2} \alpha - \frac{1}{d-2} = 0
\end{align*}
which has the roots $\alpha_1 = 1,\alpha_2 = - \frac{1}{d-2}$. Poincaré's theorem then gives us the desired result if we exclude the possibility that $a_n(\lambda) = 0$ eventually in $n$.

Assume that $N$ is such that $a_n(\lambda) = 0$ for all $n \geq N$ and $a_{N-1}(\lambda) \not=0$ (which necessarily implies $N \geq 1$). Then in particular we have 
\begin{align*}
    0 = a_{N+1}(\lambda) = A_{N-1}(\lambda) \underbrace{a_N (\lambda)}_{=0} + B_{N-1} a_{N-1}(\lambda) = B_{N-1} a_{N-1}(\lambda)
\end{align*}
and since we assumed $a_{N-1}(\lambda) \not=0$, we must have $B_{N-1}(\lambda) = 0$. So this shows that we only have to consider the cases where $B_n(\lambda) = 0$ occurs for some $n \geq 0$ and $\Real \lambda \geq 0$. By inspection, this only occurs in the cases $\ell=1,m=-1$ and $\ell=2,m=1,-2$. We have 
\begin{align*}
    B_{0,d,1,-1}(0) &= B_{0,d,1,-1}(1) = 0\\
    B_{0,d,2,1}(0) &= 0\\
    B_{0,d,2,-2}(0) &= 0.
\end{align*}
Hence we have to exclude in these cases that $a_n(\lambda) = 0$ for all $n \geq N = 1$. But if this were the case, then
\begin{align*}
    0 = a_1(\lambda) = A_{-1}(\lambda) \underbrace{a_0(\lambda)}_{=1} +B_{-1}(\lambda) \underbrace{a_{-1}(\lambda)}_{=0} = A_{-1}(\lambda).
\end{align*}
Then one checks that in these cases $A_{-1}$ does not vanish, hence excluding that $a_n(\lambda) = 0 $ eventually.
\end{proof}
If $\lim_{n \to \infty} \frac{a_{n+1}(\lambda)}{a_n(\lambda)} = - \frac{1}{d-2}$, the power series $\sum_{n = 0}^\infty a_n(\lambda )x^n$ has convergence radius $d-2$ and hence yields a $C^\infty([0,1])$ solution of Eq.~\eqref{standard Heun eq}, which is what we want to rule out. If this does not occur, then, by the previous lemma, we must have $\lim_{n \to \infty} \frac{a_{n+1}(\lambda)}{a_n(\lambda)} = 1$, i.e., the corresponding power series has convergence radius 1. Since $x =1$ is the only other singularity of Eq.~\eqref{standard Heun eq} within distance one of $x=0$, we can conclude that this solution does not belong to $C^\infty([0,1])$.
Hence, the goal is to prove that the first case never occurs.

Whenever $a_n(\lambda) \not=0$, we define $r_n(\lambda) := \frac{a_{n+1}(\lambda)}{a_n(\lambda)}$. Then one can rewrite Eq.~\eqref{a recurr rel} as
\begin{align}\label{r recurr rel}
    r_{n+1}(\lambda) = A_n(\lambda) + \frac{B_n(\lambda)}{r_n(\lambda)}
\end{align}
for $n \geq 0$ if $r_n(\lambda) \not=0$. The initial condition is given by
\begin{align*}
    r_0(\lambda) = \frac{a_1(\lambda)}{ a_0(\lambda)} = A_{-1}(\lambda) a_0(\lambda) + B_{-1}(\lambda) a_{-1}(\lambda) = A_{-1}(\lambda).
\end{align*}
\subsection{Quasi-solutions}
In this formulation we aim to prove that $r_n(\lambda)$ is well-defined eventually in $n$ and $\lim_{n \to \infty} r_n(\lambda) = 1$. For this, we will construct a so-called \emph{quasi-solution} $\Tilde{r}_n$ that is supposed to approximate $r_n$ sufficiently well.

In order to do this, first note that $r_n(\lambda)$ is a rational function in $\lambda$, where the degrees of the numerator and the denominator differ by 2. So we expect that the behavior for large $\lambda$ is similar to a second degree polynomial in $\lambda$. Now one can check (for the general case) that 
\begin{align}\label{lambda asymp of rn}
    r_n(\lambda) = \frac{\lambda^2}{2(n+1)(2n +2\ell + d)} + \frac{(4n + 2\ell +1 )\lambda}{2(n+1)(2n +2\ell +d)} + O_{n,d,\ell}(1),\quad \lambda \to \infty.
\end{align}
Similarly, one has 
\begin{align*}
    r_n(\lambda) = \frac{\ell}{4(n+1)} + O_{n,d,\lambda}(1),\quad \ell \to \infty.
\end{align*}
Hence a reasonable guess would be to set
\begin{align*}
    \Tilde{r}_n(\lambda) =  \frac{\lambda^2}{2(n+1)(2n +2\ell + d)} + \frac{(4n + 2\ell +1 )\lambda}{2(n+1)(2n +2\ell +d)} + \frac{\ell}{4(n+1)} + E_{d,\ell,m}(n),
\end{align*}
where the error term $E_{d,\ell,m}(n)$ satisfies
\begin{align*}
    \lim_{n \to \infty} E_{d,\ell,m}(n) = 1.
\end{align*}
This turns out to be somewhat successful after some modifications, at least as long as either $d \leq 5$ or $\ell \leq 2$.

To this end, let us first consider the special cases
\begin{align*}
    \Tilde{r}_{n,d,\ell,m}(\lambda) := \begin{cases}
        \frac{\lambda^2}{2(n+1)(2n + 8)} + \frac{(4n + 5)\lambda}{2(n+1)(2n+8)} + \frac{4n +5}{2(2n + 7)},&d=4,\ell=1,m=1+4-2\\
        \frac{\lambda^2}{2(n+1)(2n + d +4)} + \frac{(4n + 5)\lambda}{2(n+1)(2n + d +4)} + \frac{4n + 5}{2(2n + d + 4)}, &d\geq 5,\ell=1, m=1+d-2\\
        \frac{\lambda^2}{2(n+1)(2n + 8)} + \frac{(4n + 5)\lambda}{2(n+1)(2n + 8)} + \frac{20n + 33}{10(2n + 9)}, &d=4,\ell=1,m=1\\
        \frac{\lambda^2}{2(n+1)(2n+d+4)} + \frac{(4n+5)\lambda}{ 2(n+1)(2n+d+4)} + \frac{2n+3}{2n + d +5}, &d\geq 5,\ell=1,m=1\\
        \frac{\lambda^2}{2(n+1)(2n + 10)} + \frac{(4n + 7)\lambda}{2(n+1)(2n + 10)} + \frac{2n + 3}{2n +7},&d=4,\ell=2,m=2+d-2\\
        \frac{\lambda^2}{2(n+1)(2n + d + 6)} + \frac{(4n + 7)\lambda}{2(n+1)(2n+d+6)} + \frac{2n + 5}{2n + d + 6},&d\geq 5,\ell=2,m=2+d-2
    \end{cases}.
\end{align*}
Then, in the general case with $\ell \leq 2$
\begin{align*}
    \Tilde{r}_{n,d,\ell,m}(\lambda) := \begin{cases}
        \frac{\lambda^2}{2(n+1)(2n + d + 2)} + \frac{(4n+4)\lambda}{2(n+1)(2n + d +2)} + \frac{2n + 1}{2n + d + 2}, &d \geq 4,\ell=1, m=-1\\
        \frac{\lambda^2}{2(n+1)(2n+d+4)} + \frac{(4n + 6)\lambda}{2(n+1)(2n + d + 4)} + \frac{2n + 2}{2n + d +4},&d \geq 4,\ell=2,m=1\\
        \frac{\lambda^2}{2(n+1)(2n + d + 4)} + \frac{(4n + 6)\lambda}{2(n+1)(2n + d +4)} + \frac{2n + 3}{2n + d + 4},&d \geq 4,\ell=2,m=-2
    \end{cases}.
\end{align*}
Finally, the cases $d=4,5$ and $\ell \geq 3$
\begin{align*}
    \Tilde{r}_{n,4,\ell,\ell+d-2}(\lambda) &:= \frac{\lambda^2}{2(n+1)(2n + 2\ell + 4)} + \left( \frac{4n + 2\ell + 1}{2(n+1)(2n+2\ell+4)} + \frac{1}{2(n+1)(2n+4)} \right)\lambda \\
    &+\frac{3\ell -1}{2(7n+3)} + \frac{n-1}{n+1} \\
    \Tilde{r}_{n,4,\ell,1}(\lambda) &:= \frac{\lambda^2}{2(n+1)(2n + 2\ell + 4)} + \left( \frac{4n + 2\ell + 1}{2(n+1)(2n+2\ell+4)} + \frac{1}{2(n+1)(2n+4)} \right)\lambda \\
    &+\frac{3\ell }{2(7n+3)} + \frac{n-1}{n+1} \\
    \Tilde{r}_{n,4,\ell,-\ell}(\lambda) &:= \frac{\lambda^2}{2(n+1)(2n + 2\ell + 4)} + \left( \frac{4n + 2\ell + 1}{2(n+1)(2n+2\ell+4)} + \frac{1}{2(n+1)(2n+4)} \right)\lambda \\
    &+\frac{3\ell +2}{2(7n+3)} + \frac{n-1}{n+1}\\
    \Tilde{r}_{n,5,\ell,\ell+d-2}(\lambda) &:= \frac{\lambda^2}{2(n+1)(2n + 2\ell + 5)} + \left( \frac{4n + 2\ell + 1}{2(n+1)(2n+2\ell+5)} + \frac{1}{2(n+1)(2n+5)} \right)\lambda \\
    &+\frac{3\ell -1}{2(7n+3)} + \frac{n-1}{n+1}\\
    \Tilde{r}_{n,5,\ell,1}(\lambda) &:= \frac{\lambda^2}{2(n+1)(2n + 2\ell + 5)} + \left( \frac{4n + 2\ell + 1}{2(n+1)(2n+2\ell+5)} + \frac{1}{2(n+1)(2n+5)} \right)\lambda \\
    &+\frac{3\ell }{2(7n+3)} + \frac{n-1}{n+1}\\
    \Tilde{r}_{n,5,\ell,-\ell}(\lambda) &:= \Tilde{r}_{n,5,\ell,1}(\lambda) .
\end{align*}
Observe that in these cases there was only one additional parameter, either $d$ \emph{or} $\ell$, but not both at the same time. These expressions were found in a similar way as in previous implementations of this method. One starts with \eqref{lambda asymp of rn} and then tries to find a suitable constant term, see \cite{GloSch21} for more details. As can be seen here for $d=4,5$ and $\ell \geq 3$, sometimes one has to slightly tweak the coefficients of $\lambda^2$ and/or $\lambda$.

Now we come to the case $d \geq 6$ and $\ell \geq 3$. Here we expect the main difficulties since we now have $d$ and $\ell$ as two additional parameters. The quasi-solution method has not been implemented for two additional parameters before, so it is not clear how to proceed. It turns out that one can choose something more natural, although more complicated to analyze. We set 
\begin{align*}
    \Tilde{r}_{n,d,\ell,m}(\lambda) &:= A_{n-1,d,\ell,m}(\lambda) + \frac{1}{d-2} - \frac{5}{10n + 3d + 10} \\
    &= \frac{\lambda^2}{2(n+1)(2n+ 2\ell +d)} +\frac{(4n +2\ell + 1)\lambda}{2(n+1)(2n+ 2\ell +d)} + 1 - \frac{5}{10n + 3d + 10}\\
    &+\frac{(-2d^2 + 2d + 20)n + (d-2)\ell^2 + (-3d + 14)\ell - 2d^2 +4d - 4m +4}{2(d-2)(n+1)(2n + 2\ell + d)}
\end{align*}
for $d\geq 6$, $\ell \geq3$. We explain later, see \autoref{finding rt}, how one comes up with the idea of using $A_{n-1}$ here. Note that the first correction term $\frac{1}{d-2}$ is needed in order to guarantee $\lim_{n \to \infty} \Tilde{r}_{n}(\lambda) = 1$. 

We define
\begin{align}\label{def N}
    N(d,\ell,m):= \begin{cases}
        4, &d=4,\ell=1,m=1+4-2\\
        2,&d\geq5,\ell=1,m=1+d-2\\
        2,&d=4,\ell=1,m=1\\
        2,&d=4,\ell=2,m=2+d-2\\
        2,&d\geq4, \ell=1,m=-1\\
        1,&\text{else}
    \end{cases}.
\end{align}
This will be the starting case for various inductive arguments.
\begin{lem}\label{rt never vanishes}
We have $\Tilde{r}_{n,d,\ell,m}(\lambda) \not=0$ for all $d\geq 4$, $\ell\geq 1$, $m \in \{-\ell,1,\ell+d-2\}$, $n \geq N(d,\ell,m)$ and $\lambda \in \overline{\Half}$.
\end{lem}
\begin{proof}
Since $\Tilde{r}_{n,d,\ell,m}(\lambda)$ is a second degree polynomial in $\lambda$, with evidently positive leading coefficient, by \autoref{Wall degree 2 and 4} we only have to check that the other two coefficients are also positive. Evidently the coefficient of $\lambda$ is always positive.

Finally, for the constant coefficient, note that always $\ell \geq 1$ and $n \geq N(d,\ell,m) \geq 1$ and hence the constant term is evidently positive, except potentially in the case $d \geq 6$ and $\ell \geq 3$.

In the case $d \geq 6$ and $\ell \geq 3$ one first computes 
\begin{align*}
    \Tilde{r}_{n,d,\ell,\ell+d-2}(0) &= \frac{1}{2(d-2)(n+1)(2n + 2\ell + d)(10n + 3d + 10)}\Big[40(d-2) n^3 \\
    &+[40(d-2)\ell + 12 d^2 + 16d + 80]n^2 \\
    &+ [10(d-2)\ell^2 + (12d^2 + 6d - 20)\ell - 4d^2 + 16d + 280 ]n\\
    &+(d-2)(3d+ 10)\ell^2 + (3d^2 - 4d + 60)\ell -22d^2 + 16d + 120  \Big].
\end{align*}
The denominator is evidently positive. If one replaces $d$ by $d+6$ and $\ell$ by $\ell+3$ in the numerator, one obtains a polynomial in $n,d,\ell$ with nonnegative coefficients and positive constant term. This yields in particular that 
\begin{align*}
    \Tilde{r}_{n,d,\ell,\ell+d-2}(0) > 0
\end{align*}
for all $n \geq 1$, $d \geq 6$ and $\ell \geq 3$. The cases $m =1$ and $m=-\ell$ follow from this, since $m \mapsto \Tilde{r}_{n,d,\ell,m}(0)$ is monotonically decreasing.
\end{proof}
Next, we prove crucial analytic properties so that we can apply tools from complex analysis. 
Here and in the following we use the notation
\[ \Half:=\{\lambda\in\C: \Real\lambda>0\}. \]
\begin{lem}\label{r and rt holo}
The functions 
\begin{align*}
    r_{N(d,\ell,m),d,\ell,m},\frac{1}{\Tilde{r}_{n,d,\ell,m}}:\overline{\Half} \to \C
\end{align*}
are continuous and holomorphic in $\Half$ for all $d\geq 4$, $\ell\geq 1$, $m \in \{-\ell,1,\ell+d-2\}$ and $n \geq N(d,\ell,m)$.
\end{lem}
\begin{proof}
These are rational functions, so it suffices to check that they do not have poles on $\overline{\Half}$. Hence the statement for $\frac{1}{\Tilde{r}_{n}}$ is exactly \autoref{rt never vanishes}.

On the other hand, by Eq.~\eqref{r recurr rel} we have
\begin{align*}
    r_{N(d,\ell,m)}(\lambda) = A_{N(d,\ell,m)-1}(\lambda) + \frac{B_{N(d,\ell,m)-1}(\lambda)}{r_{N(d,\ell,m)-1}(\lambda)},
\end{align*}
provided that $r_{N(d,\ell,m)-1}(\lambda) \not=0$. Since $A_{N(d,\ell,m)-1},B_{N(d,\ell,m)-1}$ are just polynomials, it suffices to prove that the rational function $r_{N(d,\ell,m)-1}$ does not vanish on $\overline{\Half}$. Since $r_n$ can be written as a quotient $\frac{p_n}{q_n}$, where $p_n,q_n$ are polynomials with $\deg(p_n) =  2(n+1)$ and $\deg(q_n) = 2n$, the numerator of $r_{N(d,\ell,m)-1}$ will be a polynomial of degree (at most) $2 N(d,\ell,m)$. Hence except for the case $d=4,\ell=1,m=1+4-2$, this is a polynomial of degree 2 or 4. So one can apply \autoref{Wall degree 2 and 4} and one checks in each case that the conditions are fulfilled.

For $d=4,\ell=1,m=1+4-2$, we have $N(4,1,1+4-2) = 4$, so one has to check that the numerator of
\begin{align*}
    r_3(\lambda ) &= \frac{\lambda^8 + 44 \lambda^7 + 802\lambda^6 + 7832\lambda^5 + 44497\lambda^4 + 149708 \lambda^3 + 284172 \lambda^2 + 253680\lambda + 95616}{112(\lambda^6 + 27 \lambda^5 + 277 \lambda^4 + 1341 \lambda^3 + 3202 \lambda^2 + 3744 \lambda + 768)}
\end{align*}
has no zeros on $\overline{\Half}$. We will use Wall's criterion, see \autoref{Wall}, so we have to expand 
\begin{align*}
    \frac{ 44 \lambda^7  + 7832\lambda^5  + 149708 \lambda^3  + 253680\lambda }{\lambda^8 + 44 \lambda^7 + 802\lambda^6 + 7832\lambda^5 + 44497\lambda^4 + 149708 \lambda^3 + 284172 \lambda^2 + 253680\lambda + 95616}
\end{align*}
into a continued fraction. One computes the coefficients (with notation as in \autoref{Wall})
\begin{align*}
        c_1 &= \frac{1}{44}\\
        c_2 &= \frac{11}{156}\\
        c_3 &= \frac{1352}{10691}\\
        c_4 &= \frac{1257272291}{6279558844}\\
        c_5 &= \frac{14583158015998009}{47686278024425370}\\
        c_6 &= \frac{913465855584827404205}{2012154184581448576794}\\
        c_7 &= \frac{138816450390479914710584802}{144189564446831042990725115}\\
        c_8 &= \frac{10667746155294185}{5531890170247464},
\end{align*}
which are all positive, which implies the claim.
\end{proof}
With these preparations at hand, we can define auxiliary functions, for which we will later prove bounds.
\begin{defin}
We define for $n \geq N(d,\ell,m)$
\begin{align*}
    \delta_n(\lambda) &:= \frac{r_n(\lambda)}{\Tilde{r}_{n}(\lambda)} - 1\\
    C_n(\lambda) &:= \frac{B_n(\lambda)}{\Tilde{r}_n(\lambda) \Tilde{r}_{n+1}(\lambda)}\\
    \eps_n(\lambda) &:= \frac{A_n(\lambda) \Tilde{r}_n(\lambda) + B_n(\lambda)}{\Tilde{r}_{n}(\lambda) \Tilde{r}_{n+1}(\lambda)} - 1.
\end{align*}
\end{defin}
The idea of analyzing the properties of these functions comes from \cite{CosDonXia16}.
Note that $C_n(\lambda),\eps_n(\lambda)$ are well-defined, by \autoref{rt never vanishes}, and explicit. On the other hand, $\delta_n(\lambda)$ is only well-defined if $r_n(\lambda)$ is, which itself is defined by a recursion, so $\delta_{n}(\lambda)$ is not explicit but at least inherits a recurrence relation.
\begin{lem}\label{delta recur}
Let $n \geq N(d,\ell,m)$ and $\delta_n(\lambda)$ be well-defined with $1 + \delta_n(\lambda) \not=0$. Then $\delta_{n+1}(\lambda)$ is well-defined and
\begin{align*}
    \delta_{n+1}(\lambda) &= \eps_n(\lambda) - C_n(\lambda) \frac{\delta_n(\lambda)}{1 + \delta_n(\lambda)}.
\end{align*}
\end{lem}
\begin{proof}
This is a straightforward calculation using the recurrence relation \eqref{r recurr rel}.
\end{proof}
We will prove bounds on $\delta_n$ using this recurrence relation. This also requires bounding $C_n$ and $\eps_n$, see \autoref{bounds on imag axis} below. 
\begin{rem}\label{finding rt}
Taking a closer look at $\eps_n$, we observe
\begin{align*}
    \eps_n(\lambda) &=  \frac{(A_n(\lambda) - \Tilde{r}_{n+1}(\lambda)) \Tilde{r}_n(\lambda) + B_n(\lambda) }{\Tilde{r}_n(\lambda) \Tilde{r}_{n+1}(\lambda)} = \frac{(A_n(\lambda) - \Tilde{r}_{n+1}(\lambda)) \Tilde{r}_n(\lambda) }{\Tilde{r}_n(\lambda) \Tilde{r}_{n+1}(\lambda)} + C_n(\lambda).
\end{align*}
If the first summand is small enough, then the sum is small as long as $C_n(\lambda)$ is small. Thus, the question is how one can make the first summand small. Clearly the choice $\Tilde{r}_n(\lambda) = A_{n-1}(\lambda)$ would be the best in that regard. As mentioned earlier, one expects $A_{n-1}(\lambda) + \frac{1}{d-2}$ to be a better approximation, since this now converges to 1, as $n \to \infty$.

With this starting point, one makes the ansatz 
\begin{align*}
    \Tilde{r}_n(\lambda) = A_{n-1}(\lambda) + \frac{1}{d-2} + E_{n,d,\ell,m}
\end{align*}
and one finds, in a surprisingly straightforward fashion, that
\begin{align*}
    E_{n,d,\ell,m} = - \frac{5}{10n + 3d + 10}
\end{align*}
works, at least for $d \geq 6$ and $\ell \geq 3$.
\end{rem}
\begin{lem}\label{bounds on imag axis}
The bounds
\begin{align*}
    |\delta_{N(d,\ell,m)}(it )| &\leq \alpha_{d,\ell,m}\\
    |C_{n}(i t )| &\leq \beta_{n,d,\ell,m}\\
    |\eps_{n}(i t)| &\leq \gamma_{n,d,\ell,m}
\end{align*}
hold for all $d \geq 4$, $\ell \geq 1$, $m \in \{-\ell,1,\ell+d-2\}$, $n \geq N(d,\ell,m)$ and $t \in \R$, where
\begin{align}\label{choice of bounds}
    (\alpha_{d,\ell,m},\beta_{n,d,\ell,m},\gamma_{n,d,\ell,m}) &:=\begin{cases}
        \left( \frac{1}{3},\frac{1}{2},\frac{1}{12} \right), & \begin{cases}
            d \geq 4, \ell=1,m=1+4-2\\
            d \geq 5, \ell=1,m=1\\
            d \geq 4, \ell=2,m=2+d-2\\
            d = 4, \ell=1,m=-1\\
            d \geq 4, \ell=2,m=1\\
            d=4,\ell=2,m=-2
        \end{cases}\\
        \left( \frac{1}{4},\frac{11}{20},\frac{1}{15} \right), & d=4,\ell=1,m=1\\
        \left( \frac{1}{3},\frac{1}{3},\frac{1}{6} \right), &\begin{cases}
            d \geq 5,\ell=1,m=-1\\
            d \geq 5,\ell=2,m=-2
        \end{cases}\\
        \left( \frac{1}{3},\frac{1}{2} - \frac{\ell}{6(\ell+n+1)},\frac{1}{12} + \frac{\ell}{12(\ell+n+1)} \right), &d= 4,\ell \geq 3\\
        \left( \frac{1}{2},\frac{1}{3},\frac{1}{6} \right), &d\geq 5,\ell\geq3
    \end{cases}.
\end{align}
\end{lem}
\begin{proof}
The functions $\delta_{N(d,\ell,m)},C_n,\eps_n$ are rational functions in $\lambda$ with real coefficients. So one can write them as $\frac{P}{Q}$, where $P,Q$ are polynomials with real coefficients. Note that for $a,b > 0$ the bound
\begin{align*}
    \left| \frac{P(i t)}{Q(i t)} \right| \leq \frac{a}{b} \quad\forall t \in\R
\end{align*}
is equivalent to 
\begin{align*}
    a^2 |Q(i t)|^2 - b^2 |P(i t)|^2\geq 0,\quad \forall t \in \R.
\end{align*}
Observe that both $|P(i t)|^2$ and $|Q(it)|^2$ are even in $t$ and hence must be polynomials in $t^2$. Hence a sufficient condition for such an inequality to hold is that the coefficients of $a^2 |Q(i t)|^2 - b^2 |P(it)|^2$ are nonnegative.

We have the additional problem that the coefficients will depend on $n,d,\ell,m$. For $m$ one can go through each of the three cases $m = -\ell,1,\ell+d-2$ separately, which makes the coefficients only depend on $n,d,\ell$. These coefficients are now themselves polynomials in $n,d,\ell$. Then one shifts the parameter to cover the range one wants to check. So for example, in the general case one wants to check a bound for all $n \geq 1$, $d \geq 6$ and $ \ell\geq 3$, hence one replaces $(n,d,\ell)$ by $(n+1,d+6,\ell+3)$. The coefficients of \emph{that} polynomial then are all nonnegative, which proves the desired bound.

If one goes through all of the cases as described above, one will find that all coefficients are nonnegative integers and hence the bounds hold, except for $d=4,\ell=1,m=-1$. When checking the bound $|\delta_{2,4,1,-1}(it)| \leq \frac{1}{3}$, one obtains the polynomial
\begin{align*}
    1843200 - 380160 t^2 + 606252 t^4 + 75391 t^6 + 2987 t^8 + 89 t^{10} + t^{12}.
\end{align*}
Even though there appears a minus sign here, one checks that already
\begin{align*}
    1843200 - 380160 t^2 + 606252 t^4 \geq 0
\end{align*}
for all $t \in \R$, and hence the corresponding bound also holds.

All of the here occurring polynomials in the various cases can be found in the attached files, see \autoref{Additional data files}.
\end{proof}
We can extend these bounds to $\overline{\Half}$.
\begin{lem}\label{bounds on half plane}
The bounds from \autoref{bounds on imag axis} can be extended to $\overline{\Half}$.
\begin{align*}
    |\delta_{N(d,\ell,m)}(\lambda)| &\leq \alpha_{d,\ell,m}\\
    |C_n(\lambda)| &\leq \beta_{n,d,\ell,m}\\
    |\eps_n(\lambda)| &\leq \gamma_{n,d,\ell,m}
\end{align*}
for all $\lambda \in \overline{\Half}$.
\end{lem}
\begin{proof}
We know that $\delta_{N(d,\ell,m)},C_n,\eps_n$ are rational functions and by \autoref{bounds on imag axis} the bounds hold on the imaginary axis. Then \autoref{r and rt holo} implies that also $\delta_{N(d,\ell,m)},C_n,\eps_n$ are well-defined, continuous on $\overline{\Half}$ and holomorphic in $\Half$. This precisely means that they have no poles on $\overline{\Half}$. However, in this case they are trivially bounded polynomially and in particular exponentially. Applying the \nameref{phragmen lindelöf} yields the claim.
\end{proof}
Now we can finally prove the desired bound.
\begin{prop}
The bound
\begin{align}\label{delta small}
    |\delta_n(\lambda)| \leq \alpha_{d,\ell,m}
\end{align}
holds for all $d \geq 4$, $\ell \geq 1$, $m \in \{-\ell,1,\ell+d-2\}$, $n \geq N(d,\ell,m)$ and $\lambda \in \overline{\Half}$.
\end{prop}
\begin{proof}
We prove \eqref{delta small} by induction over $n$. The base case $n = N(d,\ell,m)$ is \autoref{bounds on half plane}.

Now assume that \eqref{delta small} holds for some $n \geq N(d,\ell,m)$. Since $\alpha_{d,\ell,m} \leq \frac{1}{2}$, we conclude 
\begin{align*}
    |1 + \delta_{n}(\lambda)| \geq 1 - |\delta_n(\lambda)| \geq 1 - \frac{1}{2} = \frac{1}{2}
\end{align*}
and in particular this means $1 + \delta_n(\lambda) \not=0$. Then \autoref{delta recur} implies that $\delta_{n+1}(\lambda)$ is well-defined and satisfies
\begin{align*}
    \delta_{n+1}(\lambda) = \eps_n(\lambda) - C_n(\lambda) \frac{\delta_n(\lambda)}{1 + \delta_n(\lambda)}.
\end{align*}
By using the induction hypothesis and \autoref{bounds on half plane}, we get
\begin{align*}
    |\delta_{n+1}(\lambda)| &\leq |\eps_n(\lambda)| + |C_n(\lambda)| \left|\frac{\delta_n(\lambda)}{1 + \delta_n(\lambda)}\right|\leq |\eps_n(\lambda)| + |C_n(\lambda)| \frac{|\delta_n(\lambda)|}{1 - |\delta_n(\lambda)|}\\
    &\leq \gamma_{n,d,\ell,m} + \beta_{n,d,\ell,m}\frac{\alpha_{d,\ell,m}}{1 - \alpha_{d,\ell,m}} = \alpha_{d,\ell,m},
\end{align*}
where the last equality is readily verified for the various choices of $(\alpha_{d,\ell,m},\beta_{n,d,\ell,m},\gamma_{n,d,\ell,m})$ in \eqref{choice of bounds}.
\end{proof}
This bound is good enough to conclude the desired asymptotic behavior of $r_n$.
\begin{prop}
Let $d \geq 4$, $\ell \geq 1$, $m \in \{-\ell,1,\ell+d-2\}$ and $\lambda \in \overline{\Half}$. Then 
\begin{align*}
    \lim_{n \to \infty} r_{n}(\lambda) &= 1.
\end{align*}
\end{prop}
\begin{proof}
We know from \autoref{ratios converge} that $(r_n(\lambda))_n$ converges and 
\begin{align*}
    \lim_{n \to \infty} r_n(\lambda) \in \left\{  1 , -\frac{1}{d-2} \right\}.
\end{align*}
Assume that $\lim_{n \to \infty} r_{n}(\lambda) = - \frac{1}{d-2}$. Then, since $\lim_{n \to \infty}\Tilde{r}_n(\lambda) = 1$, this would imply 
\begin{align*}
    \lim_{n \to \infty } \delta_n(\lambda) &= \lim_{n \to \infty} \frac{r_n(\lambda)}{\Tilde{r}_n(\lambda)} -1 = -\frac{1}{d-2} - 1 = - \frac{d-1}{d-2}
\end{align*}
and in particular
\begin{align*}
    \lim_{n \to \infty}|\delta_n(\lambda)| = \frac{d-1}{d-2} \geq 1.
\end{align*}
This contradicts \eqref{delta small} since
\begin{align*}
    |\delta_n(\lambda)| \leq \alpha_{d,\ell,m} \leq \frac{1}{2}
\end{align*}
holds for all $n \geq N(d,\ell,m)$. Hence, we must have $\lim_{n \to \infty} r_{n}(\lambda) = 1$.
\end{proof}
This proposition together with \autoref{l0 case} proves \autoref{mode stability}.
\begin{rem}
In the case $d \geq 6$, $\ell \geq 3$ we have
\begin{align*}
    (\alpha_{d,\ell,m},\beta_{n,d,\ell,m},\gamma_{n,d,\ell,m}) = \left( \frac{1}{2},\frac{1}{3},\frac{1}{6} \right),
\end{align*}
which are constant. We find it surprising that such a choice is possible in the presence of the two parameters $d$ and $\ell$. In previous implementations of the quasi-solution method, where one additional parameter was present (see e.g. \cite{CosDonGlo17,GloSch21,CsoGloSch24,WeiKocDon25}), the triple $(\alpha,\beta_n,\gamma_n)$ often contained expressions that are rational functions in the corresponding parameter set.
We emphasize this because a successful implementation of the quasi-solution method not only requires a good choice for the quasi-solution $\Tilde{r}_n$ but also a sensible choice for $(\alpha,\beta_n,\gamma_n)$.
\end{rem}

\appendix
\section{Preliminary results}
\subsection{Poincaré's theorem on difference equations}
Consider the difference equation
\begin{align}\label{app diff eq}
    x(n+k) + p_1(n) x(n+k-1) + \ldots + p_k(n)x(n) = 0
\end{align}
with variable coefficients $p_k(n)$ and assume that for $1 \leq i \leq k$ there exist real numbers $p_i$ such that 
\begin{align}\label{app limit assum}
    \lim_{n \to \infty} p_i(n) = p_i.
\end{align}
The characteristic equation associated with \eqref{app diff eq} is
\begin{align}\label{app char eq}
    \alpha^k + p_1 \alpha^{k-1} + \ldots+ p_k =0.
\end{align}
\begin{thm}[Poincaré]\label{poincare}
Suppose that \eqref{app limit assum} is satisfied and that the roots $\alpha_1,\ldots,\alpha_k$ of Eq.~\eqref{app char eq} have distinct moduli. If $x(n)$ is a solution of \eqref{app diff eq}, then either $x(n) =0$ for all large enough $n$ or
\begin{align*}
    \lim_{n \to \infty}  \frac{x(n+1)}{x(n)} = \alpha_i,
\end{align*}
for some $i \in \{1,\ldots,k\}$.
\end{thm}
The proof and further results on difference equations can be found in \cite{Ela05}.
\subsection{Wall's criterion}
We state Wall's formulation of the Routh-Hurwitz stability criterion.
\begin{thm}[\cite{Wall45}]\label{Wall}
Let $P(z) = z^n + b_1 z^{n-1} + \ldots b_n$ be a polynomial with real coefficients and let $Q(z) = b_1 z^{n-1} + b_3 z^{n-3} + \ldots$ be the polynomial that is comprised of the odd-indexed terms of $P(z)$. Then all the zeros of $P(z)$ have negative real parts if and only if
\begin{align*}
    \frac{Q(z)}{P(z)} = \cfrac{1}{1 + c_1 z + \cfrac{1}{c_2 z + \cfrac{1}{c_3z + \genfrac{}{}{0pt}{0}{}{\ddots + \cfrac{1}{c_n z}}}}},
\end{align*}
where all $c_1,c_2,\ldots,a_n$ are positive.
\end{thm}
We mostly need the following two special cases.
\begin{cor}\label{Wall degree 2 and 4}
\begin{enumerate}[(i)]
    \item Let $b_1, b_2 \in \R$. Then all the zeros of $P(z) = z^2 + b_1 z + b_2$ have negative real parts if and only if $b_1,b_2 >0$.
    \item Let $b_1,b_2,b_3,b_4 \in \R$. Then all the zeros of $P(z)=z^4 + b_1z^3 + b_2 z^2 + b_3 z + b_4$ have negative real parts if and only if $b_1,b_3,b_4 > 0$ and $b_1 b_2 b_3 - b_3^2 - b_1^2 b_4 > 0$.
\end{enumerate}
\end{cor}
\begin{proof}
$(i):$ One calculates
\begin{align*}
    \frac{b_1z}{z^2 + b_1 z + b_2} = \frac{1}{1 + c_1 z + \frac{1}{c_2 z}},
\end{align*}
where 
\begin{align*}
    c_1 &= \frac{1}{b_1}\\
    c_2 &= \frac{b_2}{b_1},
\end{align*}
which clearly are both positive if and only if both $b_1,b_2$ are positive. Hence Wall's criterion yields the claim.

$(ii):$ One calculates
\begin{align*}
    \frac{b_1 z^3 + b_3 z}{z^4 + b_1 z^3 + b_2 z^2 + b_3 z + b_4} &=  \frac{1}{1 + c_1 z + \frac{1}{c_2 z + \frac{1}{c_3 z + \frac{1}{c_4 z}}}}
\end{align*}
with 
\begin{align*}
    c_1 &= \frac{1}{b_1}\\
    c_2 &= \frac{b_1^2}{b_1 b_2 - b_3}\\
    c_3 &= \frac{(b_1 b_2 - b_3)^2}{b_1(b_1 b_2 b_3 - b_3^2 - b_1^2 b_4)}\\
    c_4 &= \frac{b_1 b_2 b_3 - b_3^2 - b_1^2 b_4 }{(b_1 b_2 - b_3)b_4}.
\end{align*}
Assume first that $c_1,c_2,c_3,c_4$ are positive. Then $b_1 = \frac{1}{c_1}$ is positive. Then also $b_1 b_2 - b_3 = \frac{b_1^2}{c_2}$ is positive. This yields 
\begin{align*}
    b_1 b_2 b_3 - b_3^2 - b_1^2 b_4 &= \frac{(b_1b_2 -b_3)^2}{b_1 c_3} > 0.
\end{align*}
From this one has 
\begin{align*}
    b_4 = \frac{b_1 b_2 b_3 - b_3^2 - b_1^2 b_4}{(b_1 b_2 - b_3) c_4 } > 0
\end{align*}
and
\begin{align*}
    b_3 = \frac{b_1 b_2  b_3 - b_3^2 - b_1^2 b_4}{b_1 b_2 - b_3} + \frac{b_1^2 b_4}{b_1 b_2 - b_3} > 0.
\end{align*}
On the other hand, if $b_1,b_3,b_4$ and $b_1 b_2 b_3 - b_3^2 - b_1^2 b_4$ are positive, then we have
\begin{align*}
    b_1 b_2 - b_3 = \frac{b_1 b_2 b_3 - b_3^2 - b_1^2 b_4}{b_3} + \frac{b_1^2 b_4}{b_3} > 0.
\end{align*}
Using this, by inspection all of $c_1,c_2,c_3,c_4$ are positive. Thus Wall's criterion yields the claim.
\end{proof}
\subsection{Phragm\'en-Lindelöf principle}
\begin{thm}[Phragm\'en-Lindelöf principle]\label{phragmen lindelöf}
Let $f :\overline{\Half} \to \C$ be continuous and $f|_\Half $ holomorphic. Let $M \geq 0$. If
\begin{enumerate}
    \item $|f(it)| \leq M$ for all $t \in \R$,
    \item there exists a $C \geq 0$ such that $|f(z)| \leq C e^{|z|^\frac{1}{2}}$ for all $z \in \Half$ 
\end{enumerate}
then
\begin{align*}
    |f(z)| \leq M
\end{align*}
for all $z \in \overline{\Half}$.
\end{thm}
For a proof see for example \cite{Don24}.
\section{Lie algebra theory}
\label{Decoupling of the equations}
We want to give some insight how one decouples Eq.~\eqref{full mode eq} for $\ell \geq 1$ and $d \geq 3$. We do not really need the case $d =3$ but we include it here, also to make the approach in \cite{WeiKocDon25} more explicit.

The matrix elements of the coupling operator $K$, which acts on $\C^d$-valued functions, are (up to a pre-factor) given by
\begin{align*}
    (K_{jk}f)(\xi) := (\xi_j \pd_{\xi^k} - \xi_k \pd_{\xi^j})f(\xi),
\end{align*}
where $f: \R^d\to\C$.
The operator $K$ is an angular differential operator in the sense that it maps radial functions to 0. Consequently, it suffices to study how it acts on functions $f:\S^{d-1} \to \C^d$. The crucial idea, observed in \cite{WeiKocDon25}, is to connect this operator with representations of
\begin{align*}
    \mfk{so}(d) := \{ X \in M_d(\C): X^t = -X \},
\end{align*}
the Lie algebra of all skew-symmetric complex $(d\times d)$-matrices.
First we fix some notation. Let $E_{jk}$ be the matrix with a 1 in the $j$-th row and $k$-th column and 0 in all other entries, i.e., $(E_{jk})_{m n} = \delta_{jm} \delta_{k n}$. Then setting $F_{jk} := E_{j k}  - E_{k j}$, we get that $\{F_{jk}\}_{1 \leq j < k \leq d}$ forms a basis of $\mfk{so}(d)$. Now consider the polynomial space 
\begin{align*}
    \Y_\ell := \{ p \in \C[\xi] : \Delta p = 0  \text{ and }p\text{ is homogeneous of degree }\ell\}
\end{align*}
and denote by $\Y_\ell^d$ the space of $\C^d$-valued functions where each component is in $\Y_\ell$.
We obtain a representation $\psi_\ell$ of $\mfk{so}(d)$ on $\Y_\ell$ by defining 
\begin{align*}
    \psi_\ell(F_{j k}) p := K_{jk} p.
\end{align*}
On the other hand, one has the standard representation $\pi: \mfk{so}(d) \to \mfk{gl}(\C^d)$ simply given by $\pi(X) = X$.

Now one defines the representation $\rho_\ell:\mfk{so}(d) \to \mfk{gl}(\Y_\ell^d)$ by
\begin{align*}
    (\rho_\ell(X) \Vec{p})_k := X_{k}^{\ j} \Vec{p}_j + \psi_\ell(X)\Vec{p}_k = (\pi(X) \Vec{p})_k + \psi_\ell(X) \Vec{p}_k.
\end{align*}
The reason we are interested in this representation comes from the identity
\begin{align}\label{casimir of rho}
    \sum_{1 \leq j < k \leq d} \rho_{\ell}(F_{j k})^2 = - (d-1)\id_{\Y_\ell^d} + 2 K|_{\Y_\ell^d} - \ell(\ell+d-2)\id_{\Y_\ell^d},
\end{align}
which is easily checked. Since we aim at diagonalizing $K|_{\Y_\ell^d}$, it suffices to do so for $\sum_{1 \leq j < k \leq d} \rho_{\ell}(F_{j k})^2$. It turns out that this is a (multiple of a) \emph{Casimir operator}. Since Casimir operators always act diagonally on irreducible subrepresentations, our goal will be to find the irreducible components of $(\Y_\ell^d,\rho_\ell)$.

We will outline in the following how to find this decomposition. For this we will make constant use of standard definitions and theorems of Lie algebra theory, which can be found for example in \cite{Hum78,FulHar91}, but we try to be as explicit as possible.

In \cite{WeiKocDon25} the same is achieved in the special case $d=3$ by using the fact that the representation $(\Y_\ell^d,\rho_\ell)$ is isomorphic to the tensor product of the representations $(\C^d,\pi)$ and $(\Y_\ell, \psi_\ell)$ and using classical results on angular momentum operators in three dimensions, see for example \cite[Proposition 17.23]{Hal13}.
\subsection{Passing to an equivalent bilinear form}
The following ideas come from \cite[§18.1]{FulHar91}.

The Lie algebra $\mfk{so}(d)$ is the space of matrices that are skew-symmetric with respect to the standard bilinear form $b$ on $\C^d$, which is given by 
\begin{align*}
    b(z,w) := \sum_{j = 1}^d z_j w_j.
\end{align*}
It is more convenient to work with an equivalent bilinear form on $\C^d$. If $\Tilde{b}$ is another bilinear form on $\C^d$, we define the space
\begin{align*}
    \mfk{so}(\Tilde{b}) := \{X \in M_d(\C):  \Tilde{b}(X z , w) = - \Tilde{b}(z, X w) \quad \forall z,w \in \C^d\}.
\end{align*}
We will consider
\begin{align*}
    \Tilde{b}_{d}(z,w):=\Tilde{b}(z,w) := \begin{cases}
        \sum_{j=1}^{n} (z_j w_{n+j} + z_{n+j} w_j), & d=2n \text{ is even}\\
        \sum_{j=1}^{n} (z_j w_{n+j} + z_{n+j} w_j) + z_{2n+  1}w_{2n+ 1}, & d=2n+1 \text{ is odd}
    \end{cases}.
\end{align*}
We will have to distinguish between $d$ odd or even at times. This form is related to the standard form via 
\begin{align*}
    \Tilde{b}(z , w) = b( Tz,T w),\quad \forall z,w\in\C^d,
\end{align*}
where
\begin{align*}
    T = \begin{cases}
        \begin{pmatrix}
            I_n & \frac{1}{2} I_n\\
            i I_n & -\frac{1}{2} i I_n
        \end{pmatrix},& d = 2n\\
        \begin{pmatrix}
            I_n & \frac{1}{2} I_n & 0\\
            i I_n & -\frac{1}{2} i I_n & 0\\
            0 & 0 & 1
        \end{pmatrix},& d=2n+1
    \end{cases}
\end{align*}
and $I_n$ is the $(n\times n)$-identity matrix. Then $\mfk{so}(\Tilde{b})$ is a Lie algebra (again with the commutator as a Lie bracket) isomorphic to $\mfk{so}(d)$ with Lie algebra isomorphism $\phi:\mfk{so}(d) \to \mfk{so}(\Tilde{b})$
\begin{align*}
    \phi(X) = T^{-1} X T.
\end{align*}
We prefer to work with $\mfk{so}(\Tilde{b})$, since its elements have a nice block structure. Namely $X \in \mfk{so}(\Tilde{b})$ if and only if
\begin{align*}
    X = \begin{cases}
        \begin{pmatrix}
            A & B\\
            C & -A^t
        \end{pmatrix}, \quad A \in M_n(\C), B,C \in \mfk{so}(n) &\text{if } d = 2n\\
        \begin{pmatrix}
            A & B & z_1\\
            C & -A^t &z_2\\
            -z_2^t & -z_1^t & 0
        \end{pmatrix}, \quad A \in M_n(\C), B,C \in \mfk{so}(n), z_1,z_2 \in \C^n &\text{if } d = 2n+1
    \end{cases},
\end{align*}
see \cite[p.~269]{FulHar91}.
Note that we can consider $\mfk{so}(\Tilde{b}_{2n})$ as a subspace of $\mfk{so}(\Tilde{b}_{2n+1})$.
\subsection{Cartan subalgebra and root space decomposition}
The main advantage of using this bilinear form comes from the fact that the Cartan subalgebra and the roots are then more natural. We define
\begin{align*}
    H_j := \begin{pmatrix}
        E_{j j} & 0\\
        0 & -E_{j j}
    \end{pmatrix} = \diag(E_{j j}, - E_{j j}), \quad 1 \leq j \leq n.
\end{align*}
Then $\mfk{h} = \text{span}\{H_j:1 \leq j \leq n\} \subseteq \mfk{so}(\Tilde{b}_{2n}) \subseteq \mfk{so}(\Tilde{b}_{2n+1})$, the subspace of diagonal matrices, is a Cartan subalgebra both of $\mfk{so}(\Tilde{b}_{2n})$ and $\mfk{so}(\Tilde{b}_{2n+1})$. We define $\{L_j\}_{j=1}^n \subseteq \mfk{h}^\ast$ to be the dual basis of $\{H_j\}_{j=1}^n$, so $L_j (H_k) = \delta_{j k}$. Then the root system $R$ is given by
\begin{align*}
    R = \begin{cases}
        \{\pm L_j \pm L_k: 1 \leq j < k \leq n\}, & d=2n\\
        \{\pm L_j \pm L_k: 1 \leq j < k \leq n\} \cup \{\pm L_j : 1 \leq j \leq n\}, & d=2n+1
    \end{cases}
\end{align*}
and the root spaces are spanned by 
\begin{align*}
    L_j - L_k:& \begin{pmatrix}
        E_{j k} & 0 \\
        0 &- E_{k j}
    \end{pmatrix}\\
    L_j + L_k:& \begin{pmatrix}
        0 & E_{j k} - E_{k j}\\
        0 & 0
    \end{pmatrix}\\
    -L_j - L_k:& \begin{pmatrix}
        0 & 0 \\
        E_{j k} - E_{k j} & 0
    \end{pmatrix}
\end{align*}
for both $d$ even and odd and for $d$ odd one additionally has
\begin{align*}
    L_j: &\begin{pmatrix}
        0 & 0 & e_j\\
        0 & 0 & 0\\
        0 & - e_j^t &0
    \end{pmatrix}\\
    -L_j: &\begin{pmatrix}
        0 & 0 & 0\\
        0 & 0 & e_j\\
        -e_j^t & 0 &0
    \end{pmatrix},
\end{align*}
where $e_j$ denotes the $j$-th unit vector in $\C^d$. We choose as the positive roots 
\begin{align*}
    R^+ = \begin{cases}
        \{ L_j \pm L_k: 1 \leq j < k \leq n\}, & d=2n\\
        \{ L_j \pm L_k: 1 \leq j < k \leq n\} \cup \{ L_j : 1 \leq j \leq n\}, & d=2n+1
    \end{cases}.
\end{align*}
\subsection{Irreducible representations}
A finite dimensional irreducible representation of a complex semisimple Lie algebra is classified by its highest weight, which is a dominant integral weight, see \cite[p.~113]{Hum78}. In our case a functional $L = \sum_{j=1}^n a_j L_j \in \mfk{h}^\ast$ is a dominant integral weight if and only if
\begin{itemize}
    \item either all $a_j$ are integers or all of them are half-integers, i.e., elements of $\frac{1}{2} \Z \setminus \Z$ and
    \item the inequalities
    \begin{align*}
        \begin{cases}
            a_1 \geq \ldots \geq a_{n-1} \geq |a_n| & d=2n\\
            a_1 \geq \ldots  \geq a_{n-1} \geq a_n \geq 0 &d=2n+1
        \end{cases}
    \end{align*}
    hold.
\end{itemize}
In this case the, up to isomorphism unique, representation is denoted by $\Gamma(L)$. For us only the integer case will be relevant.

In order for us to efficiently identify an irreducible component of our representation at hand, we state a useful lemma.
\begin{lem}\label{decomp lemma}
Let $\mfk{g}$ be a (finite dimensional) complex semisimple Lie algebra with a chosen Cartan subalgebra $\mfk{h}$ and system of positive roots $R^+$. Denote for $\alpha \in R^+$ the corresponding root space by $\mfk{g}_\alpha$.

Let $(V,\rho)$ be a representation of $\mfk{g}$. Assume that $0 \not=v \in V$ satisfies the following.
\begin{itemize}
    \item $\rho(h) v = L(h) v$ for all $h \in \mfk{h}$, where $L \in \mfk{h}^\ast$ is a dominant integral weight.
    \item $\rho(x) v = 0$ for all $x \in \mfk{g}_\alpha$ and $\alpha \in R^+$.
\end{itemize}
Then $W$, the subrepresentation of $V$ generated by $v$, is finite dimensional, irreducible and has the highest weight $L$ with highest weight vector $v$.
\end{lem}
Sometimes this condition for $v$ is taken as a definition for being a highest weight vector, for example in \cite[p.~202]{FulHar91}.

Note that the representation $\rho_\ell : \mfk{so}(d) \to \mfk{gl}(\Y_\ell^d)$ induces a representation $\Tilde{\rho}_\ell : \mfk{so}(\Tilde{b}) \to \mfk{gl}(\Y_\ell^d)$ by defining $\Tilde{\rho}_\ell := \rho_\ell \circ \phi^{-1}$. For $\Tilde{\rho}_\ell$ one now has to find the right element $ \Vec{p} \in \Y_\ell^d$ and the right dominant integral weight $L$ to satisfy the conditions of \autoref{decomp lemma}.

We will explicitly write down $L$ and $\Vec{p}$ for each case.
\begin{itemize}
    \item $d=3$: We have $L = (\ell+1)L_1$ with 
    \begin{align*}
        \Vec{p} = \begin{pmatrix}
            (\xi_1 + i \xi_2)^\ell\\
            i (\xi_1 + i \xi_2)^\ell\\
            0
        \end{pmatrix},
    \end{align*}
    then $L = \ell L_1$ with
    \begin{align*}
        \Vec{p} = \begin{pmatrix}
            - (\xi_1 + i \xi_2)^{\ell-1} \xi_3\\
            - i (\xi_1 + i \xi_2)^{\ell-1}\xi_3\\
            (\xi_1 + i \xi_2)^{\ell}
        \end{pmatrix}
    \end{align*}
    and $L = (\ell-1)L_1$ with
    \begin{align*}
        \Vec{p} = \begin{pmatrix}
            (\xi_1 + i \xi_2)^{\ell-2}(\mu |\xi|^2 + i(\xi_1^2 + \xi_2^2)) + i(\xi_1 + i \xi_2)^\ell\\
            i(\xi_1 + i \xi_2)^{\ell-2}(\mu |\xi|^2 + i (\xi_1^2 + \xi_2^2)) + (\xi_1 + i \xi_2)^{\ell}\\
            2 i (\xi_1 + i \xi_2)^{\ell-1} \xi_3
        \end{pmatrix},
    \end{align*}
    where $\mu = -\frac{2(\ell-1) i}{2\ell - 1}$.
    \item $d=4$: We have $L = (\ell+1)L_1$ with
    \begin{align*}
        \Vec{p} = \begin{pmatrix}
            (\xi_1 + i \xi_3)^\ell\\
            0\\
            i(\xi_1 + i \xi_3)^\ell\\
            0
        \end{pmatrix},
    \end{align*}
    then $L = \ell L_1 + L_2$ with
    \begin{align*}
        \Vec{p} = \begin{pmatrix}
            (\xi_1 + i \xi_3)^{\ell-1}(\xi_2 + i \xi_4)\\
            -(\xi_1 + i \xi_3)^\ell\\
            i(\xi_1 + i \xi_3)^{\ell-1}(\xi_2 + i \xi_4)\\
            - i (\xi_1 + i \xi_3)^\ell
        \end{pmatrix},
    \end{align*}
    then $L = \ell L_1 - L_2$ with 
    \begin{align*}
        \Vec{p} = \begin{pmatrix}
        (\xi_1 + i \xi_3)^{\ell-1}(\xi_2 - i\xi_4)\\
        -(\xi_1 + i \xi_3)^\ell\\
        i(\xi_1 + i \xi_3)^{\ell-1}(\xi_2 - i \xi_4)\\
        i (\xi_1 + i \xi_3)^\ell
    \end{pmatrix}
    \end{align*}
    and $L = (\ell-1)L_1$ with
    \begin{align*}
        \Vec{p} = \begin{pmatrix}
            (\xi_1 + i \xi_3)^{\ell-2}(\mu |\xi|^2 + i (\xi_1^2 + \xi_3^2)) + i (\xi_1 + i \xi_3)^\ell\\
            2 i (\xi_1 + i \xi_3)^{\ell-1}\xi_2\\
            i(\xi_1 + i \xi_3)^{\ell-2}(\mu |\xi|^2 + i(\xi_1^2 + \xi_3^2)) + (\xi_1 + i \xi_3)^\ell\\
            2 i (\xi_1 + i \xi_3)^{\ell-1}\xi_4
        \end{pmatrix},
    \end{align*}
    where $\mu = - \frac{(\ell-1)i}{\ell}$.
    \item $d  \geq 5$: The following choices work for both $d =2n$ even and $d = 2n+1$ odd. We have $L = (\ell+1)L_1$ with
    \begin{align*}
        \Vec{p}_j = \begin{cases}
            (\xi_1 + i \xi_{n+1})^\ell, & j=1\\
            i (\xi_1 + i \xi_{n+1})^\ell, & j=n+1\\
            0, & \text{else}
        \end{cases},
    \end{align*}
    then $L = \ell L_1 + L_2$ with
    \begin{align*}
        \Vec{p}_j = \begin{cases}
            (\xi_1 + i \xi_{n+1})^{\ell-1}(\xi_2 + i \xi_{n+2}),  & j=1\\
             - (\xi_1 + i \xi_{n+1})^\ell, & j=2\\
            i(\xi_1 + i \xi_{n+1})^{\ell-1}(\xi_2 + i \xi_{n+2}),  & j=n+1\\
             - i (\xi_1 + i \xi_{n+1})^\ell, & j=n+2\\
            0, & \text{else}
        \end{cases}
    \end{align*}
    and $L = (\ell-1)L_1$ with
    \begin{align*}
        \Vec{p}_j = \begin{cases}
            (\xi_1 + i\xi_{n+1})^{\ell-2} (\mu |\xi|^2 + i(\xi_1^2 + \xi_{n+1}^2)) + i (\xi_1 + i \xi_{n+1})^\ell, & j=1\\
            i(\xi_1 + i \xi_{n+1})^{\ell-2}(\mu |\xi|^2 + i (\xi_1^2 + \xi_{n+1}^2) ) + (\xi_1 + i \xi_{n+1})^\ell, & j=n+1\\
            2i(\xi_1 + i \xi_{n+1})^{\ell-1}\xi_j, & \text{else}
        \end{cases},
    \end{align*}
    where $\mu = - \frac{2(\ell-1)i}{2\ell + d - 4}$.
\end{itemize}
Note that for the weight $L = (\ell-1) L_1$ when $\ell = 1$, one has a negative power but this cancels with a positive power to yield a polynomial in the end. For example for $d =3$ and $\ell=1$ one has
\begin{align*}
    \Vec{p} = \begin{pmatrix}
        (\xi_1 + i \xi_2)^{-1} i (\xi_1^2 + \xi_2^2) + i(\xi_1 + i \xi_2)\\
        i(\xi_1 + i \xi_2)^{-1} i (\xi_1^2 + \xi_2^2) + \xi_1 + i\xi_2\\
        2i \xi_3
    \end{pmatrix} = \begin{pmatrix}
        i(\xi_1 - i \xi_2) + i(\xi_1 + i \xi_2)\\
         - (\xi_1 - i \xi_2) + \xi_1 + i \xi_2\\
         2 i \xi_3
    \end{pmatrix} = 2i \xi \in \Y_1^3.
\end{align*}
For illustration, we show that $L$ and $\Vec{p}$ satisfy the conditions of \autoref{decomp lemma} for the case $d = 3$ and $\ell = 1$. First of all, we have $L = a_1 L_1$ for $a_1 = 0,1,2$ which evidently satisfy $a_1 \geq 0$, hence $L$ is a dominant integral weight. Next, we recall that $\Tilde{\rho}_1 = \rho_1 \circ \phi^{-1}$, where 
\begin{align*}
    \phi^{-1}(X) = T X T^{-1} = \begin{pmatrix}
        1 & \frac{1}{2} &0\\
        i & -\frac{1}{2} i &0\\
        0 & 0 & 1
    \end{pmatrix}X \begin{pmatrix}
        \frac{1}{2} & - \frac{1}{2}i &0\\
        1 & i &0\\
        0 & 0 &1
    \end{pmatrix}.
\end{align*}
Since $ d = 3 = 2 \cdot 1 + 1$, we have $n = 1$ and hence $\mfk{h} = \text{span} \{H_1\}$ where
\begin{align*}
    H_1 = \begin{pmatrix}
        1 & 0 & 0\\
        0 & -1 & 0\\
        0 & 0 &0
    \end{pmatrix}.
\end{align*}
In particular, the only positive root $L_1 \in \mfk{h}^\ast $ is given simply by $L_1(H_1) = 1$. The root space of this root is spanned by 
\begin{align*}
    X_1 := \begin{pmatrix}
        0 & 0& 1\\
        0 & 0 & 0\\
        0 & -1 & 0
    \end{pmatrix}.
\end{align*}
Now we compute
\begin{align*}
    \Tilde{\rho}_1(H_1) &= i (E_{2 1} - E_{1 2}) + i (\xi_2 \pd_{\xi^1} - \xi_1 \pd_{\xi^2})\\
    \Tilde{\rho}_1(X_1) &= E_{1 3}  - E_{3 1} + i (E_{2 3} - E_{3 2}) + \xi_1 \pd_{\xi^3} - \xi_3 \pd_{\xi^1} + i(\xi_2 \pd_{\xi^3} - \xi_3 \pd_{\xi^2}),
\end{align*}
where the differential operators act diagonally. Now we just insert $\Vec{p}$ in the various cases and compute the relations.
\begin{itemize}
    \item $L= 2 L_1$: Here we have $\Vec{p} = \begin{pmatrix}
        \xi_1 + i \xi_2\\
        i(\xi_1 + i \xi_2)\\
        0
    \end{pmatrix}$ and we compute
    \begin{align*}
        \Tilde{\rho}_1(H_1) \Vec{p} &= \left( \begin{pmatrix}
            0 & -i & 0\\
            i & 0  & 0\\
            0 & 0 & 0
        \end{pmatrix} + i (\xi_2 \pd_{\xi^1} - \xi_1 \pd_{\xi^2}) \right)\begin{pmatrix}
        \xi_1 + i \xi_2\\
        i(\xi_1 + i \xi_2)\\
        0
    \end{pmatrix}\\
    &= \begin{pmatrix}
        \xi_1 + i \xi_2\\
        i(\xi_1 + i \xi_2)\\
        0
    \end{pmatrix} + \begin{pmatrix}
        i(\xi_2 - i\xi_1)\\
        - (\xi_2 - i \xi_1)\\
        0
    \end{pmatrix} = \begin{pmatrix}
        \xi_1 + i \xi_2\\
        i(\xi_1 + i \xi_2)\\
        0
    \end{pmatrix} + \begin{pmatrix}
        \xi_1 + i \xi_2\\
        i(\xi_1 + i \xi_2)\\
        0
    \end{pmatrix} = 2 \Vec{p} = L(H_1) \Vec{p}\\
    \Tilde{\rho}_1(X_1) \Vec{p} &= \left(\begin{pmatrix}
            0 & 0 & 1\\
            0 & 0 & i\\
            -1 & -i & 0
        \end{pmatrix} + \xi_1 \pd_{\xi^3} - \xi_3 \pd_{\xi^1} + i(\xi_2 \pd_{\xi^3} - \xi_3 \pd_{\xi^2})\right)\begin{pmatrix}
        \xi_1 + i \xi_2\\
        i(\xi_1 + i \xi_2)\\
        0
    \end{pmatrix}\\
    &= \begin{pmatrix}
        0\\
        0\\
        - (\xi_1 + i \xi_2) + \xi_1 + i \xi_2
    \end{pmatrix} + \begin{pmatrix}
        - \xi_3\\
        - i \xi_3\\
        0
    \end{pmatrix} + \begin{pmatrix}
        i (-i \xi_3)\\
        - (-i \xi_3)\\
        0
    \end{pmatrix}\\
    &= \begin{pmatrix}
        0\\
        0\\
        0
    \end{pmatrix} + \begin{pmatrix}
        -\xi_3\\
        -i \xi_3\\
        0
    \end{pmatrix} + \begin{pmatrix}
        \xi_3\\
        i \xi_3\\
        0
    \end{pmatrix} = 0.
    \end{align*}
    \item $L=L_1$: Here we have $\Vec{p} = \begin{pmatrix}
        -\xi_3\\
        -i\xi_3\\
        \xi_1 + i \xi_2
    \end{pmatrix}$ and we compute
    \begin{align*}
        \Tilde{\rho}_1(H_1) \Vec{p} &= \left( \begin{pmatrix}
            0 & -i & 0\\
            i & 0  & 0\\
            0 & 0 & 0
        \end{pmatrix} + i (\xi_2 \pd_{\xi^1} - \xi_1 \pd_{\xi^2}) \right)\begin{pmatrix}
            -\xi_3\\
            -i\xi_3\\
            \xi_1 + i \xi_2
        \end{pmatrix}\\
        &= \begin{pmatrix}
            - \xi_3\\
            - i \xi_3\\
            0
        \end{pmatrix} + \begin{pmatrix}
            0\\
            0\\
            i (\xi_2 - i \xi_1 ) 
        \end{pmatrix} = \begin{pmatrix}
            - \xi_3\\
            - i \xi_3\\
            0
        \end{pmatrix} + \begin{pmatrix}
            0\\
            0\\
            \xi_1 + i \xi_2
        \end{pmatrix} = \Vec{p} = L(H_1) \Vec{p}\\
        \Tilde{\rho}_1(X_1) \Vec{p} &= \left(\begin{pmatrix}
            0 & 0 & 1\\
            0 & 0 & i\\
            -1 & -i & 0
        \end{pmatrix} + \xi_1 \pd_{\xi^3} - \xi_3 \pd_{\xi^1} + i(\xi_2 \pd_{\xi^3} - \xi_3 \pd_{\xi^2})\right)\begin{pmatrix}
            -\xi_3\\
            -i\xi_3\\
            \xi_1 + i \xi_2
        \end{pmatrix}\\
        &= \begin{pmatrix}
            \xi_1 + i \xi_2\\
            i \xi_1 - \xi_2\\
            \xi_3 - \xi_3
        \end{pmatrix} + \begin{pmatrix}
            -\xi_1\\
            -i \xi_1\\
            -\xi_3
        \end{pmatrix} + \begin{pmatrix}
            -i \xi_2\\
            \xi_2\\
            \xi_3
        \end{pmatrix} = 0.
    \end{align*}
     \item $L = 0L_1$: We have $\Vec{p} = 2 i \xi = \begin{pmatrix}
        2i \xi_1\\
        2 i \xi_2\\
        2i \xi_3
    \end{pmatrix}$ and we compute
    \begin{align*}
        \Tilde{\rho}_1(H_1) \Vec{p} &= \left( \begin{pmatrix}
            0 & -i & 0\\
            i & 0  & 0\\
            0 & 0 & 0
        \end{pmatrix} + i (\xi_2 \pd_{\xi^1} - \xi_1 \pd_{\xi^2}) \right)\begin{pmatrix}
            2i \xi_1\\
            2i\xi_2\\
            2i\xi_3
        \end{pmatrix} \\
        &=\begin{pmatrix}
            2  \xi_2\\
            -2 \xi_1\\
            0
        \end{pmatrix}  + \begin{pmatrix}
            i (2i \xi_2)\\
            i (-2 i \xi_1)\\
            0
        \end{pmatrix} = \begin{pmatrix}
            2  \xi_2\\
            -2 \xi_1\\
            0
        \end{pmatrix}  + \begin{pmatrix}
            -2 \xi_2\\
            2 \xi_1\\
            0
        \end{pmatrix} = 0 = L(H_1) \Vec{p}\\
        \Tilde{\rho}_1(X_1) \Vec{p} &= \left(\begin{pmatrix}
            0 & 0 & 1\\
            0 & 0 & i\\
            -1 & -i & 0
        \end{pmatrix} + \xi_1 \pd_{\xi^3} - \xi_3 \pd_{\xi^1} + i(\xi_2 \pd_{\xi^3} - \xi_3 \pd_{\xi^2})\right) \begin{pmatrix}
            2i \xi_1\\
            2i\xi_2\\
            2i\xi_3
        \end{pmatrix}\\
        &= \begin{pmatrix}
            2 i \xi_3\\
            -2 \xi_3\\
            -2i \xi_1 + 2 \xi_2
        \end{pmatrix} + \begin{pmatrix}
            -2 i \xi_3\\
            0\\
            2i\xi_1
        \end{pmatrix} + \begin{pmatrix}
            0\\
            i(-2 i \xi_3)\\
            i(2 i \xi_2)
        \end{pmatrix} \\
        &= \begin{pmatrix}
            2 i \xi_3\\
            -2 \xi_3\\
            -2i \xi_1 + 2 \xi_2
        \end{pmatrix} + \begin{pmatrix}
            -2 i \xi_3\\
            0\\
            2i\xi_1
        \end{pmatrix} + \begin{pmatrix}
            0\\
            2  \xi_3\\
            -2 \xi_2
        \end{pmatrix} = 0.
    \end{align*}
\end{itemize}
Hence the conditions of \autoref{decomp lemma} are satisfied in these cases.
\begin{prop}
For $d \geq 3$ and $\ell \geq 1$ the representation $(\Y_\ell^d,\rho_\ell)$ has the following decomposition into irreducible subrepresentations
\begin{align}\label{decomp of Yld}
    \Y_\ell^d \cong \begin{cases}
        \Gamma((\ell+1)L_1)\oplus \Gamma(\ell L_1) \oplus \Gamma((\ell-1)L_1), & d=3\\
        \Gamma((\ell+1)L_1)\oplus \Gamma(\ell L_1 + L_2) \oplus \Gamma(\ell L_1 - L_2) \oplus \Gamma((\ell-1)L_1), &d=4\\
        \Gamma((\ell+1)L_1) \oplus \Gamma(\ell L_1 + L_2) \oplus \Gamma((\ell-1)L_1), & d\geq 5
    \end{cases}.
\end{align}
\end{prop}
\begin{proof}
The fact that these representations appear as subrepresentations follows by applying \\
\autoref{decomp lemma} for the various cases of $(\Vec{p},L)$ listed above. Also, since the weights are different, it is clear that their sum is direct.

It remains to be shown that this sum exhausts $\Y_\ell^d$. This follows since the dimensions of the summands in the decomposition add up to the dimension of $\Y_\ell^d$, which can be checked using Weyl's dimension formula, see \cite[Corollary 24.6]{FulHar91}.

We illustrate this for $d =3$ and $\ell = 1$. Weyl's dimension formula reads
\begin{align*}
    \dim \Gamma(L) = \frac{\prod_{\alpha \in R^+}(L + \rho ,\alpha)}{\prod_{\alpha \in R^+}( \rho ,\alpha)},
\end{align*}
where $\rho := \frac{1}{2} \sum_{\alpha \in R^+}\alpha$ is the half-sum of the positive roots and $(\cdot,\cdot)$ is the inner product on $\mfk{h}^\ast$ induced by the Killing form (see \cite[p.~208]{FulHar91}). In our case we simply have $R^+ = \{L_1\}$ and hence $\rho = \frac{1}{2} L_1$. For $L = a_1 L_1$ we compute
\begin{align*}
    \dim \Gamma(L) &= \frac{\prod_{\alpha \in R^+}(L + \rho ,\alpha)}{\prod_{\alpha \in R^+}( \rho ,\alpha)} = \frac{(a_1 L_1 + \frac{1}{2} L_1,L_1)}{(\frac{1}{2} L_1,L_1)} = \frac{\left(a_1 + \frac{1}{2}\right)}{\frac{1}{2}} \frac{(L_1,L_1)}{(L_1,L_1)} = 2a_1 + 1.
\end{align*}
With this we compute
\begin{align*}
    \dim \Gamma(2 L_1) + \dim \Gamma(L_1) + \dim \Gamma(0 L_1) = 5 + 3 +1 = 9 = 3 \cdot 3 = \dim(\Y_1^3)
\end{align*}
hence the dimensions coincide as claimed.
\end{proof}
We denote by $C_{\Gamma(L)}$ the Casimir operator of $\Gamma(L)$, see \cite[p.~416]{FulHar91} for the definition. This operator just acts by a scalar $c_{\Gamma(L)}$, for which there is a formula available, see \cite[p.~418, Eq.~(25.14)]{FulHar91}. In our cases we have 
\begin{align}\label{eigenvalues of Casimir}
    c_{\Gamma(L)} = \begin{cases}
        \frac{1}{2(d - 2)}(\ell+1)(\ell+d-1), &L = (\ell+1)L_1, d \geq 3\\
        \frac{1}{2(d-2)}(\ell+1)(\ell+d-3), & \begin{cases}
            L = \ell L_1 + L_2, d\geq 4\\
            L = \ell L_1 - L_2, d=4\\
            L = \ell L_1, d=3
        \end{cases}\\
        \frac{1}{2(d-2)}(\ell-1)(\ell+d-3), & L=(\ell-1)L_1, d\geq 3
    \end{cases}.
\end{align}
We again illustrate this for $d =3$ and $\ell =1$. Eq.~$(25.14)$ from \cite[p.~418]{FulHar91} reads
\begin{align*}
    c_{\Gamma(L)} = (L + \rho,L + \rho ) - (\rho,\rho).
\end{align*}
In our case we have $\rho = \frac{1}{2} L_1$ and we compute for $L = a_1 L_1$
\begin{align*}
    c_{\Gamma(L)} &= (L + \rho,L + \rho) - (\rho,\rho) = \left(a_1 L_1 + \frac{1}{2}L_1,a_1 L_1 + \frac{1}{2} L_1\right) - \left(\frac{1}{2}L_1,\frac{1}{2}L_1\right)\\
    &= \left(\left(a_1 + \frac{1}{2}\right)^2 - \frac{1}{4}\right)(L_1,L_1) = a_1(a_1 + 1)(L_1,L_1).
\end{align*}
We have the normalization $(L_1,L_1) = \frac{1}{2}$ and we compute
\begin{align*}
    c_{\Gamma(2L_1)} &= \frac{2 \cdot 3}{2} = 3 \\
    c_{\Gamma(L_1)} &= \frac{1 \cdot 2}{2} = 1\\
    c_{\Gamma(0 L_1)} &= \frac{0 \cdot 1}{2} = 0,
\end{align*}
which are exactly the values as claimed in \eqref{eigenvalues of Casimir}.

Continuing with the general case, we can use this to find the value by which $\sum_{1 \leq j < k \leq d} \rho_\ell(F_{j k})^2$ acts on the irreducible components. One has to be a bit careful with normalization here, since the Killing form $\kappa$ on $\mfk{so}(d)$ is a multiple of the trace form
\begin{align*}
    \kappa(X_1,X_2) = (d-2) \tr(X_1 X_2),
\end{align*}
which yields that the dual basis of $\{F_{j k}\}_{1 \leq j < k \leq d}$ with respect to the Killing form is given by $\{-\frac{1}{2(d-2)}F_{j k}\}_{1 \leq j < k \leq  d}$. With this in mind, we write
\begin{align*}
    \sum_{1 \leq j < k \leq d} \rho_\ell(F_{j k})^2 &= -2(d-2) \sum_{1 \leq j < k \leq d}\rho_\ell \left( -\frac{1}{2(d-2)}F_{j k} \right)\rho_\ell(F_{j k})
\end{align*}
and with the previous we conclude that (with the abuse of notation $\Gamma(L )\subseteq \Y_l^d$)
\begin{align*}
    \left.\sum_{1 \leq j < k \leq d} \rho_\ell(F_{j k})^2\right|_{\Gamma(L)} = -2(d-2) c_{\Gamma(L)} \id_{\Gamma(L)}.
\end{align*}
Finally, inserting this into Eq.~\eqref{casimir of rho} yields
\begin{align*}
    K|_{\Gamma(L)} = \begin{cases}
        -\ell \cdot\id_{\Gamma(L)}, & L = (\ell+1)L_1,d\geq 3\\
        1 \cdot\id_{\Gamma(L)}, & \begin{cases}
            L = \ell L_1 + L_2, d \geq 4\\
            L = \ell L_1 - L_2,d =4\\
            L = \ell L_1, d=3
        \end{cases}\\
        (\ell+d-2)\cdot \id_{\Gamma(L)}, & L = (\ell-1)L_1, d \geq 3
    \end{cases}
\end{align*}
and we conclude that $K|_{\Y_\ell^d}$ has exactly the eigenvalues $-\ell,1,\ell+d-2$.

Note that this holds for all $\ell \geq 1$ and $d \geq 3$. It is somewhat surprising that this result is in the end uniform in $d$. From the perspective of Lie algebra theory one would expect some special cases. Indeed, in our derivation of these eigenvalues, we used the decomposition \eqref{decomp of Yld} which has the special cases $d=3,4$. It then seems to be a ``coincidence'' that the Casimir operator (and hence also $K|_{\Y_\ell^d}$ by Eq.~\eqref{casimir of rho}) has exactly three different eigenvalues, see \eqref{eigenvalues of Casimir}.

This leads us to believe that there probably exists some way of obtaining these eigenvalues without using Lie algebra theory.
\section{Additional data files}
\label{Additional data files}
\noindent Since the polynomials appearing in the proof of \autoref{bounds on imag axis} are quite large, we find it sensible to attach them as the following csv-files
\begin{align*}
    {\tt 1\_min\_4.csv, 1\_min\_geq5.csv, 1\_1\_4.csv, 1\_1\_geq5.csv, 1\_pl\_4.csv, 1\_pl\_geq5.csv}\\
    {\tt 2\_min\_4.csv, 2\_min\_geq5.csv, 2\_1\_geq4.csv, 2\_pl\_4.csv, 2\_pl\_geq5.csv }\\
    {\tt geq3\_min\_4.csv, geq3\_min\_5.csv, geq3\_min\_6.csv, geq3\_1\_4.csv, geq3\_1\_5.csv, geq3\_1\_geq6.csv }\\
    {\tt geq3\_pl\_4.csv, geq3\_pl\_5.csv, geq3\_pl\_geq6.csv }.
\end{align*}
These are named as {\tt l\_m\_d.csv} where {\tt l} describes the range of $\ell$ that is considered in that file. Analogously {\tt d} describes the range of $d $ considered. Finally {\tt m} has the values min,1,pl which correspond to the values $-\ell,1,\ell+d-2$ of $m$, respectively. The content of each file are two columns with the following entries.
\renewcommand{\arraystretch}{1.2}
\begin{longtable}{|l|l|}

\caption{Description of the content of the data files}\\
\hline Variable name & Description  \\ \hline 

    {\tt A} & \makecell[l]{Explicit expression of the rational function $A_{n,d,\ell,m}(\lambda)$ as defined in \eqref{An special cases}\\ 
    and \eqref{An general case}.} \\
    {\tt B} & Explicit expression of the rational function $B_{n,d,\ell,m}(\lambda)$ as defined in \eqref{Bn all cases}. \\
    {\tt N} & $N(d,\ell,m)$ as defined in \eqref{def N}.\\
    {\tt d\_0} & If a range of $d$ is considered, then this is the starting value.\\
    {\tt l\_0} & If a range of $\ell$ is considered, then this is the starting value.\\
    {\tt r\_N} & Explicit expression of the rational function $r_{N(d,\ell,m),d,\ell,m}(\lambda)$.\\
    {\tt rtilde} & Explicit expression of the rational function $\Tilde{r}_{n,d,\ell,m}(\lambda)$.\\
    {\tt delta\_N} & Explicit expression of the rational function $\delta_{N(d,\ell,m),d,\ell,m}(\lambda)$.\\
    {\tt C} & Explicit expression of the rational function $C_{n,d,\ell,m}(\lambda)$.\\
    {\tt epsilon} & Explicit expression of the rational function $\eps_{n,d,\ell,m}(\lambda)$.\\ 
    {\tt bounddelta} & \makecell[l]{This is the polynomial $a^2|Q(it)|^2 - b^2|P(it)|^2$ as described in the proof of \\
    \autoref{bounds on imag axis}. It is a polynomial in $t^2$ and, depending on the case, also in \\
    $d$ and/or $\ell$ with integer coefficients. If $d$ and/or $\ell$ appear, they are shifted\\ by $d_0$ and $\ell_0$, respectively. }\\
    {\tt boundC} & \makecell[l]{This is the polynomial analogous to {\tt bounddelta}, but with $n$ as an\\
    additional variable. Here $n$ is shifted by $N(d,\ell,m)$.}\\
    {\tt boundepsilon} & This is the polynomial analogous to {\tt boundC}.
    \\ \hline
    
\end{longtable}

\printbibliography
\end{document}